\documentclass[a4paper]{article}

\usepackage{GC_style_spectral_gap}

\mathtoolsset{showonlyrefs}         

\usepackage[
  a4paper,
  top=2.5cm,
  bottom=2.5cm,
  left=2.5cm,
  right=2.5cm
]{geometry}

\begin{document}

	\title{\textsc{Existence, scaling, and spectral gap for travelling fronts in the 2D renormalized Allen--Cahn equation}}
    \author[1]{\textsc{Gideon Chiusole}\thanks{GC gratefully acknowledges support by the doctoral program TopMath and partial funding by the VolkswagenStiftung via a Lichtenberg Professorship awarded to Christian Kuehn. GC would also like to thank Severin Schraven for helpful discussions regarding Schrödinger operators and dilations.}\,\orcidlink{0009-0005-3930-2924}}
    \author[1]{\textsc{Christian Kuehn}\thanks{CK would like to thank the VolkswagenStiftung for support via a Lichtenberg Professorship.}\,\orcidlink{0000-0002-7063-6173}}
    \affil[1]{\textsc{Technical University of Munich, Munich, Germany.}}
	\date{}

\maketitle

\vspace*{-20pt}

\begin{abstract}
We study the deterministic skeleton of the renormalized stochastic Allen--Cahn equation in spatial dimension $2$. For all sufficiently small regularization parameters $\delta>0$, we construct monotone travelling wave front solutions connecting the renormalized equilibria, derive a small-$\delta$ asymptotic description of their profile and speed, and identify the leading-order contributions. Linearizing about the wave and working in a naturally chosen weighted space, we show that there exists a spectral gap between the symmetry induced eigenvalue $0$ and the rest of the spectrum. The spectral gap grows linearly in the renormalization constant as $\delta\downarrow 0$.
\end{abstract}




\tableofcontents
\newpage


\section{Introduction}

This paper is part of a larger effort to understand the stability of patterns in singular stochastic partial differential equations (SPDEs). As a model problem, we consider the stochastic Allen--Cahn equation with additive spacetime white noise $\xi$, which is formally given by

\begin{equation}
    \partial_t u = \Delta u \underbrace{- u(u-\alpha)(u-1)}_{=: f(u)} + \sigma \xi, \quad \alpha \in (0, 1/2), \sigma > 0, \label{eq:stochastic_AC_intro}
\end{equation}

where $u = u(t, \bx)$ is (formally) an unknown time-dependent random scalar field on the full space $\bbR^d$ which we split as $\bx = (x_1, x_{\perp})$ into a longitudinal direction $x_1 \in \bbR$ and a transversal direction $x_{\perp} \in \bbR^{d-1}$.\footnote{The content of this paper concerns the case of spatial dimension $d = 2$, but we keep $d$ general in the following preliminary discussion. While the choice of $\bbR$ in the longitudinal direction is fixed by the shape of the wave profile, the choice of $\bbR^{d-1}$ in the transversal direction is rather arbitrary. There are other reasonable choices, like the torus $\bbT^{d-1}$. As far as the results of this paper are concerned, the choice of transversal direction is irrelevant though.} 

The following two facts about this equation are well known:

\begin{enumerate}
    \item The deterministic version of \eqref{eq:stochastic_AC_intro}, given by formally setting $\sigma=0$, admits planar monotone travelling wave front solutions connecting the two locally asymptotically stable homogeneous equilibria $0$ and $1$ of the non-linearity $f$. That is, there exists a smooth function $\Phi: \mathbb{R} \rightarrow \mathbb{R}$ with 
    
    \begin{equation}
        \lim_{x \rightarrow -\infty} \Phi(x) = 1, \quad \lim_{x \rightarrow +\infty} \Phi(x) = 0, \quad \Phi'(x) < 0 \quad \text{ for all } x \in \mathbb{R}, \label{eq:TW_profile_intro}
    \end{equation}

    and $s \in \mathbb{R}$ such that $u(t,\bx) = \Phi(x_1 - st)$ is a solution of \eqref{eq:stochastic_AC_intro} with $\sigma = 0$ - see e.g. \cite[Sec. 8]{sattingerStabilityWavesNonlinear1976}, \cite[Sec. 2.3.3.2]{kapitulaSpectralDynamicalStability2013}. In the comoving frame $(\eta, x_{\perp}) := (x_1 - st, x_{\perp}) \in \bbR^d$, the planar wave profile $(\eta, x_{\perp}) \mapsto \Phi(\langle \eta, \cdot \rangle)$ is a stationary solution, and the linearization around it is given by

    \begin{equation}
        \underline{\sL} = \Delta + s \partial_{\eta} + f'(\Phi). \label{eq:linearized_AC_intro}
    \end{equation}

    It can be split into a longitudinal part and a transversal part as
    
    \begin{equation}
    \underline{\sL} = \sL \otimes I + I \otimes \Delta_{\perp}, \quad \text{with} \quad \sL = \partial_{\eta}^2 + s \partial_{\eta} + f'(\Phi), \quad \text{and} \quad \Delta_{\perp} = \sum_{i=2}^d \partial_{i}^2. \label{eq:decomp_intro}
    \end{equation}

    The spectral analysis of the operator $\underline{\sL}$ can thus be reduced to the analysis of the one-dimensional operator $\sL$. By abstract principles, due to the translational symmetry of the equation in the direction of propagation $x_1$, the operator $\sL$ has a neutral mode given by $\partial_{\eta} \Phi$, which spans the kernel of $\sL$. However, modulo this neutral mode, the travelling wave solutions described above are stable (in an appropriate sense - see \cite[Thm. 1.2]{kapitulaMultidimensionalStabilityPlanar1997}). Crucial for the proof of this statement is the existence of a spectral gap between $0$ and the rest of the spectrum of $\sL$ (see \cite[Hyp. 1.1]{kapitulaMultidimensionalStabilityPlanar1997}). That is, there exists $\beta > 0$ such that

    \begin{equation}
        \sigma(\sL) \cap \left\{ z \in \bbC : \Real(z) \geq -\beta \right\} = \{ 0 \}. \label{eq:spectral_gap_intro}
    \end{equation}
    
    \item In dimension $d \geq 2$, the stochastic Allen--Cahn Equation \eqref{eq:stochastic_AC_intro} is ill-posed due to the interplay of the irregularity of the white noise $\xi$ and the non-linearity $f$. However, in dimension $d \leq 3$ it can be renormalized to restore well-posedness. More specifically, let $\rho_{\delta}$ be a mollifier at scale $\delta > 0$ and let $\xi_{\delta} = \rho_{\delta} \ast \xi$ denote the regularization of the noise at scale $\delta$ via convolution. Then, for each $d \in \{2, 3\}$, there exists a polynomial $g_{\delta}$, with coefficients depending on the regularization scale $\delta > 0$, such that solutions to the renormalized equation
    
    \begin{equation}
        \partial_t u = \Delta u - u(u-\alpha)(u-1) + g_{\delta, \sigma}(u) + \sigma \xi_{\delta}, \label{eq:ren_AC_first}
    \end{equation}

    converge as $\delta \downarrow 0$ to a limit which is independent of the choice of mollifier $\rho_{\delta}$ - see \cite{dapratoStrongSolutionsStochastic2003, hairerTheoryRegularityStructures2014}. Solutions to Equation \eqref{eq:stochastic_AC_intro} should then be understood as the limit of the solutions to the renormalized equation. In the case of dimension $d = 2$, the renormalization polynomial $g_{\delta}$ can be obtained via Wick renormalization and is given by

    \begin{equation}
        g_{\delta}(u) = 3 C_{\delta, \sigma} u - (1+\alpha) C_{\delta, \sigma}, \label{eq:renormalization_polynomial}
    \end{equation}

    where $C_{\delta, \sigma} \simeq \sigma^2 \log(\delta^{-1})$ as $\delta \downarrow 0$.
\end{enumerate}

These two facts naturally lead us to study the deterministic skeleton 

\begin{equation}
        \partial_t u = \partial_x^2 u + f(u) + g_{\delta, \sigma}(u) = \partial_x^2 u - u^3 + (1 + \alpha) u^2 + (3 C_{\delta, \sigma} - \alpha) u - (1 + \alpha) C_{\delta, \sigma}, \label{eq:ren_AC_skeleton}
\end{equation}

of the renormalized Equation \eqref{eq:ren_AC_first} in the direction of propagation, the existence of travelling wave solutions therein, and the existence a spectral gap for the associated linearization around these waves. With a view~\cite{kuehnTravellingWavesMonostable2020,boschLocalPhaseTracking2025} towards stability results for the stochastic problem \eqref{eq:ren_AC_first} in the renormalized limit $\delta \downarrow 0$, we are not only interested in the qualitative properties of the wave and spectral gap, but also quantitative dependence in terms of the regularization parameter $\delta > 0$ and $\sigma > 0$.

\paragraph{Strategy, organization and main results}

In order to study the singular limit $\delta \downarrow 0$ and hence $C_{\delta, \sigma} \rightarrow \infty$, we begin by extracting the relevant scales of the travelling wave problem of Equation \eqref{eq:ren_AC_skeleton}, which turn out to be $\varepsilon := C_{\delta, \sigma}^{-1/2}$ and its powers. In particular, in Section \ref{sec:equilibria_waves} we compute an asymptotic expansion of the renormalized equilibria $(\zminus, \zzero, \zplus)$ in Lemma \ref{lem:three_real_solutions} and show that there exists a travelling wave solution $(\Pren, \sren)$ of Equation \eqref{eq:ren_AC_skeleton} connecting these equilibria satisfying 

\begin{equation}
    \sup_{x \in \bbR} \vert \Pren(x) \vert \simeq \varepsilon^{-2} , \quad \sren \simeq \varepsilon^2, \label{eq:wave_data_intro}
\end{equation}

in Theorem \ref{thm:wave_data}. From there, in Section \ref{sec:spectral_analysis}, we study the spectral problem as a meromorphic perturbation problem around $\varepsilon = 0$. We introduce the $\varepsilon$-dependent linearized operator $\Lren$ around the renormalized wave $(\Pren, \sren)$ in an $\varepsilon$-dependent comoving frame on an appropriately weighted space and show in Theorem \ref{thm:sigma_ess} and Theorem \ref{thm:sigma_0} that it has a spectral gap $\beta(\varepsilon)$, between $0$ and the rest of the spectrum, which behaves like

\begin{equation}
    \beta(\varepsilon) \simeq \varepsilon^{-2} = C_{\delta, \sigma}.
\end{equation}

The proof is split into two parts - see Figure \ref{fig:spectral_analysis_strategy}. In Subsection \ref{subsec:essential_spectrum} we compute the essential spectrum by conjugating $\Lren$ to an asymptotically constant operator $\Mren$ on an unweighted $L^2$-space and then decomposing $\Mren$ into a sum of an asymptotic operator $\Minfty$ and a relatively compact perturbation $\sK$. By Weyl's theorem, the contribution of $\sK$ does not affect the essential spectrum of $\Minfty$, which can be computed via Fredholm theory and the existence or absence of an exponential dichotomy for an ODE associated to $\Minfty$. In Subsection \ref{subsec:discrete_spectrum} we compute the discrete spectrum. We firstly reduce the spectral gap problem for the Sturm--Liouville operator $\Lren$ on a weighted space to a problem for a Schrödinger operator $\Hren$, by using the strong decay properties of Schrödinger eigenfunctions in classically forbidden regions (Agmon decay). We then extend the domain of the parameter $\varepsilon$ to a small punctured ball around $0$, leading to a meromorphic perturbation problem. Finally, we use dilation to reduce further to a regular perturbation problem of Schrödinger operators.

\paragraph{Notation}

Throughout the paper, function spaces are over the complex numbers. The inner product on $L^2(\bbR)$ is given by $\langle f, g \rangle_{L^2} = \int_{\bbR} \overline{f(x)} g(x) \dd x$ for $f, g \in L^2(\bbR)$. For a (complex) Banach space $X$, we denote by $\calB(X)$ the space of bounded linear operators and by $I$ the identity operator on $X$. For a closed linear operator $\sL$ on $X$ with domain $\calD(\sL)$, we denote by $\rho(\sL)$ its resolvent set and for $z \in \rho(\sL)$, the resolvent operator is given by $\calR(\sL, z) := (\sL - z I)^{-1}$. The sets $\sigma_{\ess}(\sL)$, $\sigma_0(\sL)$, and $\sigma_{\pt}(\sL)$ denote the essential spectrum and its complement (defined in \eqref{eq:sigma_ess} and \eqref{eq:sigma_0}) and the set of eigenvalues (i.e. the set of $z \in \bbC$ such that $\sL - z I: \calD(\sL) \rightarrow X$ is not injective), respectively. 

For two real numbers $A_{\varepsilon},B_{\varepsilon}$, depending on a parameter $\varepsilon$, the symbol $A_{\varepsilon} \lesssim B_{\varepsilon}$ means that there exists a constant $C > 0$, independent of $\varepsilon$, such that $A_{\varepsilon} \leq C B_{\varepsilon}$. The notation $A_{\varepsilon} \simeq B_{\varepsilon}$ means that both $A_{\varepsilon} \lesssim B_{\varepsilon}$ and $B_{\varepsilon} \lesssim A_{\varepsilon}$ hold. 

The set $B(\varepsilon_0) := \{ z \in \bbC : \vert z \vert < \varepsilon_0 \}$ denotes the open ball of radius $\varepsilon_0 > 0$ around $0$ in the complex plane and $B'(\varepsilon_0) := B(\varepsilon_0) \setminus \{ 0 \}$ the corresponding punctured ball. The set $\bbR_+ := [0, \infty)$ and $\bbR_- := (- \infty, 0]$ denote the non-negative and non-positive half axis, respectively. The real and imaginary part of $z \in \bbC$ are denoted by $\Real(z)$ and $\Ima(z)$, respectively.

Generally, conceptual type ($\Ren, \hol, \ldots$) and spatial position (positive/negative/center/asymptotic: $+,-, 0, \infty, \ldots$) will be displayed as superscript, while parameter dependence (mostly $\varepsilon$) and index (mostly $0, 1, \ldots$) will be displayed as a subscript.

\begin{figure}[htbp]
    \centering
    \includegraphics[width=\textwidth]{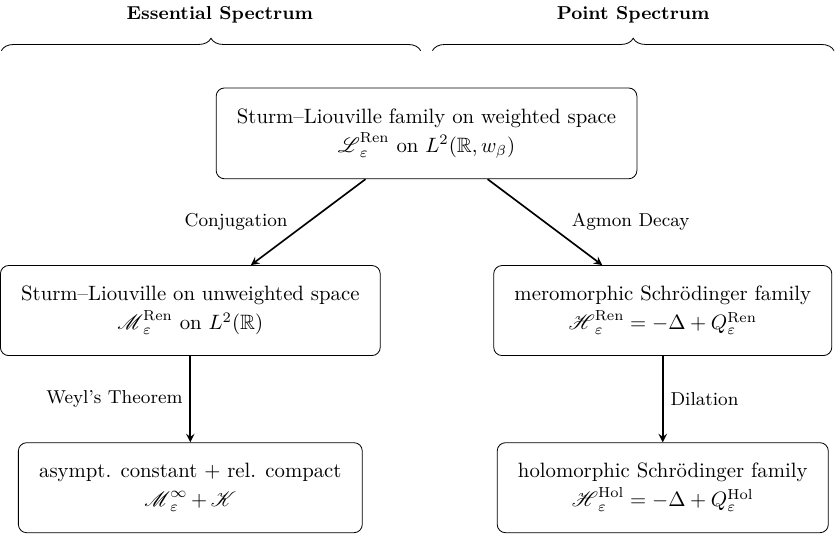} 
    \caption{Overview over the proof of the spectral gap problem.}
    \label{fig:spectral_analysis_strategy}
\end{figure}


\section{Renormalized equilibria and waves} \label{sec:equilibria_waves}

Throughout the paper, we will work around the limiting case $C_{\delta, \sigma} = \infty$. In the later sections, we want to consider this as a meromorphic perturbation problem and thus immediately switch scales to a small parameter $\varepsilon := C_{\delta, \sigma}^{-1/2}$. The reason for the choice of exponent will become apparent in the following section. The non-linearity in \eqref{eq:ren_AC_skeleton} thus takes the form

\begin{equation}
    \fren(u) := - u^3 + (1 + \alpha) u^2 + (3 \varepsilon^{-2} - \alpha) u - (1 + \alpha) \varepsilon^{-2}, \label{eq:nonlinearity}
\end{equation}

where $\alpha \in (0, 1/2)$ is fixed throughout the paper and thus suppressed in the notation. In this section we study the roots of $\fren$ in terms of $\varepsilon \downarrow 0$ in order to compute the asymptotic states of the travelling wave solutions of \eqref{eq:ren_AC_skeleton} which will be derived in Theorem \ref{thm:wave_data}. Because in Section \ref{subsec:discrete_spectrum} we will rely on holomorphic perturbation theory, we will immediately allow $\varepsilon \in \bbC$ in the following.

\begin{lem} \label{lem:three_real_solutions}
    There exists an $\varepsilon_0 > 0$ and holomorphic functions $z^-, z^0, z^+: B'(\varepsilon_0) \rightarrow \bbC$ such that
        
    \begin{equation}
        \fren(\zminus) = \fren(\zzero) = \fren(\zplus) = 0, \quad \forall \varepsilon \in B'(\varepsilon_0), \label{eq:solution}
    \end{equation}

    which are real-valued for $\varepsilon \in B'(\varepsilon_0) \cap (0, \infty)$ and satisfy

    \begin{equation}
        \zminus < \zzero < \zplus, \quad \text{for } \varepsilon \in B'(\varepsilon_0) \cap (0, \infty). \label{eq:ordering}
    \end{equation}

    The leading parts of their Laurent series around $\varepsilon = 0$ are 

    \begin{align}
        \zzero &= \frac{1 + \alpha}{3} + \calO(\varepsilon^2), \\
        \zpm &= \pm \sqrt{3} \varepsilon^{-1} + \frac{1 + \alpha}{3} + \calO(\varepsilon). \label{eq:z_pm_laurent}
    \end{align}

    In particular, the expansion of $\zzero$ has no linear contribution.
\end{lem}
\begin{proof}
    In order to see that $\fren$ has three distinct real solutions for $\varepsilon$ small enough and real, recall that $\fren$ can be brought into the standard form of a depressed cubic $\fren(z) = y^3 + b_1 y + b_0$ via the variable/parameter change
    
    \begin{align}
        y &:= - \left(z - \frac{1+\alpha}{3} \right), \label{eq:transformation_y} \\
        b_1 &:= -3 \varepsilon^{-2} + \frac{3 \alpha - (1+\alpha)^2}{3}, \label{eq:asymptotic_b1} \\
        b_0 &:= \frac{2(1+\alpha)^3 - 9 \alpha (1+\alpha)}{27}, \label{eq:asymptotic_b0}
    \end{align}
    
    see \cite[Sec. 14.6, p. 611]{dummitAbstractAlgebra2004}. Choosing $\varepsilon_0^{(1)} > 0$ small enough, the discriminant $\Disc(p_{\varepsilon}(y)) = - 4 b_1^3 - 27 b_0^2$ of $p_{\varepsilon}(y) := y^3 + b_1 y + b_0$, is non-zero for every $\varepsilon \in B'(\varepsilon_0^{(1)})$ and positive for every $\varepsilon \in B'(\varepsilon_0^{(1)}) \cap (0, \infty)$. Hence by \cite[Sec. 4.4, Bem. 3]{boschAlgebra2020} the depressed cubic $p_{\varepsilon}$ has three distinct complex solutions for every $\varepsilon \in B'(\varepsilon_0^{(1)})$ and three distinct real solutions for every $\varepsilon \in B'(\varepsilon_0^{(1)}) \cap (0, \infty)$ and the same is thus true for $\fren$.
    
    For the bounded solution branch $\varepsilon \mapsto \zzero$, consider the polynomial function 
    
    \begin{equation}
        g^{0}_{\varepsilon}(z) := \varepsilon^2 \fren(z) = - \varepsilon^2 z^3 + \varepsilon^2 (1 + \alpha) z^2 + (3 - \varepsilon^2 \alpha) z - (1 + \alpha), \label{eq:scaled_nonlin}
    \end{equation}
    
    for $\varepsilon \in \bbC$ and $z \in \bbC$ and note that for $\varepsilon = 0$ the polynomial $g^{0}_0(z) = 3z - (1 + \alpha)$ has a single simple root at $z = (1 + \alpha)/3$. Thus by Lemma \ref{lem:existence_of_holo_sol} there exists $\varepsilon_0^{(2)} > 0$ and a holomorphic function $\zzero: B(\varepsilon_0^{(2)}) \rightarrow \bbC$ such that $g^0_{\varepsilon}(\zzero) = 0$ for every $\varepsilon \in B(\varepsilon_0^{(2)})$. Since we may write the coefficients of $g^0_{\varepsilon}$ as holomorphic functions in $\varepsilon^2$ (instead of just $\varepsilon$), the Taylor expansion of $\zzero$ around $\varepsilon = 0$ only contains even powers of $\varepsilon$ and thus has the form 

    \begin{equation}
        \zzero = \frac{1 + \alpha}{3} + \calO(\varepsilon^2). \label{eq:laurent_z0}
    \end{equation}

    Since the solutions of $g^0_{\varepsilon}(z) = 0$ and $\fren(z) = 0$ coincide for every $\varepsilon \neq 0$, the restriction of $\zzero$ to $B'(\varepsilon_0^{(2)})$ is holomorphic in $\varepsilon$ and satisfies \eqref{eq:solution} for every $\varepsilon \in B'(\varepsilon_0^{(2)})$.

    For the two unbounded branches $\varepsilon \mapsto \zpm$, consider the polynomial 

   \begin{equation}
        g^{\pm}_{\varepsilon}(z) := \varepsilon^3 \fren(\varepsilon^{-1} z) = - z^3 + \varepsilon (1 + \alpha) z^2 + (3 - \alpha \varepsilon^2) z - \varepsilon (1 + \alpha). \label{eq:unbounded_branch_poly}
    \end{equation}
    
    For $\varepsilon = 0$, this polynomial reduces to $g^{\pm}_0(z) = - z^3 + 3z$, which has three distinct solutions $\{0, \pm \sqrt{3}\}$. Hence, by Lemma \ref{lem:existence_of_holo_sol} there exists $\varepsilon_0^{(3)}>0$ and holomorphic functions $\tilde{z}_1, \tilde{z}_2, \tilde{z}_3: B(\varepsilon_0^{(3)}) \rightarrow \bbC$ such that $g^{\pm}_{\varepsilon}(\tilde{z}_k(\varepsilon)) = 0$ for every $k = 1, 2, 3$ and $\varepsilon \in B(\varepsilon_0^{(3)})$. Therefore $\varepsilon \mapsto z_k(\varepsilon) := \varepsilon^{-1} \tilde{z}_k(\varepsilon)$ are solutions of $\fren(z) = 0$ and holomorphic on $B'(\varepsilon_0^{(3)})$ and hence satisfy \eqref{eq:solution}. By the above, their Laurent series around $\varepsilon = 0$ have the form

    \begin{equation}
        \zpm = \pm \sqrt{3} \varepsilon^{-1} + c_0 \varepsilon^0 + \calO(\varepsilon). \label{eq:laurent_xpm}
    \end{equation}

    To obtain the coefficient at order $\varepsilon^0$ in \eqref{eq:laurent_xpm}, we plug the ansatz \eqref{eq:laurent_xpm} for the Laurent series into $\fren(\zplus) = 0$. The resulting equation at order $\varepsilon^{-2}$ reads 

    \begin{equation}
        0 = \alpha (\pm \sqrt{3})^2 - \alpha - 3 c_0 (\pm \sqrt{3})^2 + 3 c_0 + (\pm \sqrt{3})^2 - 1 = 2 \alpha - 6 c_0 + 2,
    \end{equation}

    which is solved by $c_0 = (1 + \alpha)/3$. 

    Finally, set $\varepsilon_0 := \min \{ \varepsilon_0^{(1)}, \varepsilon_0^{(2)}, \varepsilon_0^{(3)} \}$ and note that by choosing $\varepsilon_0$ smaller if necessary, we may ensure that the ordering \eqref{eq:ordering} holds true by continuity of the functions $\zminus, \zzero, \zplus$ on $B'(\varepsilon_0)$.
\end{proof}

Recall that the unrenormalized Allen--Cahn equation admits travelling wave solutions in dimension $1$ (and thus any dimension). More precisely, for any $\alpha \in (0,1/2)$ the equation 

\begin{equation}
    \partial_t u = \partial_x^2 u - u(u-\alpha)(u-1) \label{eq:huxley}
\end{equation}

is solved by $u(t,x) := \Phux( x - \shux t)$, where 

\begin{equation}
    \Phux(x) := \frac{1}{2} \left( 1 + \tanh \left( \frac{\sqrt{2}}{4} x \right) \right) = \left( 1 + \exp\left(-x/\sqrt{2}\right) \right)^{-1}, \quad \shux = \sqrt{2} \left( \alpha - \frac{1}{2} \right). \label{eq:huxley_quant}
\end{equation}

see e.g. \cite[Sec. 8]{sattingerStabilityWavesNonlinear1976}. We proceed by showing that the renormalized skeleton equation of the Allen--Cahn Equation \eqref{eq:ren_AC_skeleton} also admits such travelling wave solutions, by transforming it into the form of \eqref{eq:huxley}.

\begin{thm} \label{thm:wave_data}
    Let $\varepsilon_0 > 0$ and $\zminus, \zzero, \zplus$ be as in Lemma \ref{lem:three_real_solutions} and let $\varepsilon \in B'(\varepsilon_0) \cap(0, \infty)$ be arbitrary and define 

    \begin{equation}
        \Aren := \zplus - \zminus = \sqrt{12} \varepsilon^{-1} + \calO(\varepsilon). \label{eq:Aren}
    \end{equation}
    
    Then Equation \eqref{eq:ren_AC_skeleton} admits travelling wave solutions with shape and speed given by

    \begin{equation}
        \Pren(x) = \Aren \Phux(\Aren x) + \zminus, \quad \sren = \calO(\varepsilon^2), \label{eq:ren_quantities}
    \end{equation}

    satisfying $\lim_{x \rightarrow \pm \infty} \Pren(x) = \zpm$.
\end{thm}
\begin{proof}
    We will transform Equation \eqref{eq:ren_AC_skeleton} to the standard form in \eqref{eq:huxley}. To this end, for any $\varepsilon \in B'(\varepsilon_0) \cap (0, \infty)$ and $x \in \bbR$, define

    \begin{align}
        \ahux &:= \frac{\zzero - \zminus}{\zplus - \zminus} \label{eq:alpha_eps}\\
        \fhux(x) &:= \Aren^{-3} \left(\fren(\Aren x + \zminus)\right). \label{eq:f_tilde}
    \end{align}

    Notice that $\fhux$ is a polynomial of order $3$ with leading coefficient $- \Aren^3 \Aren^{-3} = -1$ which has roots at $0, \ahux$, and $1$ (corresponding to the roots $\zminus, \zzero, \zplus$ of $\fren$). Therefore 

    \begin{equation}
        \fhux(u) = -u(u-\ahux)(u-1), \label{eq:f_tilde_explicit}
    \end{equation}

    which has the form of Equation \eqref{eq:huxley}. Now, for any solution $v: \bbR \times \bbR \rightarrow \bbR$ of 

    \begin{equation}
        \partial_t v = \partial_x^2 v + \fhux(v), 
    \end{equation}

    define 

    \begin{equation}
        u(t,x) := \Aren v(\Aren^{2} t, \Aren x) + \zminus
    \end{equation}

    and observe that 

    \begin{align}
        \partial_t u(t, x) &= \Aren^3 (\partial_t v)(\Aren^{2} t, \Aren x) \\
        &= \Aren^3 (\partial_x^2 v)(\Aren^{2} t, \Aren x) + \Aren^3 \fhux(v(\Aren^{2} t, \Aren x)) \\
        &= \Aren \partial_x^2 (v(\Aren^{2} t, \Aren x) + \zminus) + \Aren^3 \fhux(v(\Aren^{2} t, \Aren x)) \\
        &= \partial_x^2 u(t,x) + \fren(\Aren v(\Aren^{2} t, \Aren x) + \zminus) \\
        &= \partial_x^2 u(t,x) + \fren(u(t,x)).
    \end{align}

    Therefore, since \eqref{eq:huxley} is solved by $(t,x) \mapsto \Phux(x - \shux t)$, Equation \eqref{eq:ren_AC_skeleton} is solved by 

    \begin{equation}
        (t,x) \mapsto u(t,x) = \Aren v(\Aren^{2} t, \Aren x) + \zminus = \Aren \Phux(\Aren x - \Aren^{2} \shux t) + \zminus,
    \end{equation}
    
    which is a travelling wave with profile and speed given by 

    \begin{equation}
        \Pren(x) = \Aren \Phux(\Aren x) + \zminus, \quad \sren = \Aren \shux = \Aren \sqrt{2} \left( \ahux - \frac{1}{2} \right).
    \end{equation}
    
    The asymptotic expansion of $\Aren$ follows directly from \eqref{eq:z_pm_laurent}. For the asymptotic expansion of $\sren$, notice firstly that

    \begin{equation}
        \ahux = \frac{\frac{1+\alpha}{3} + \calO(\varepsilon^2) + \sqrt{3}\varepsilon^{-1} - \frac{1+\alpha}{3} + \calO(\varepsilon)}{\sqrt{3}\varepsilon^{-1} + \frac{1+\alpha}{3} + \calO(\varepsilon) + \sqrt{3}\varepsilon^{-1} - \frac{1+\alpha}{3} + \calO(\varepsilon)} = \frac{\sqrt{3}\varepsilon^{-1}(1 + \calO(\varepsilon^2))}{\sqrt{3}\varepsilon^{-1}(2 + \calO(\varepsilon^2))} = \frac{1 + \calO(\varepsilon^2)}{2 + \calO(\varepsilon^2)}.
    \end{equation}
    
    and that $\ahux$ is real analytic on $B'(\varepsilon_0) \cap(0, \infty)$ because $\zminus, \zzero, \zplus$ are holomorphic on $B'(\varepsilon_0)$ by Lemma \ref{lem:three_real_solutions}. Since its denominator has non-zero constant part, $\ahux$ has a real-analytic extension across $0$ and its Taylor series around $\varepsilon = 0$ must satisfy 

    \begin{equation}
        (c_0 + c_1 \varepsilon + \calO(\varepsilon^2))(2 + \calO(\varepsilon^2)) = 1 + \calO(\varepsilon^2),
    \end{equation}

    and hence $c_0 = 1/2$ and $2 c_1 = 0$. In particular, $\ahux = 1/2 + \calO(\varepsilon^2)$ and thus $\sren = \sqrt{2}(\ahux - 1/2) = \calO(\varepsilon^2)$. 

    Finally, the asymptotic values $\lim_{x \rightarrow \pm \infty} \Pren(x) = \zpm$ follow from $\Aren > 0$ (due to \eqref{eq:ordering}) and the asymptotic values of $\Phux$ being $0$ and $1$.
\end{proof}

\begin{rem}
    By computing the Laurent expansions in Lemma \ref{lem:three_real_solutions} to a higher degree, it can be shown that the speed $\sren$ is actually of order $\calO(\varepsilon^3)$ and negative for $\varepsilon$ real, positive, and small enough. However, this will not be relevant for the rest of the paper.
\end{rem}


\section{Spectral analysis} \label{sec:spectral_analysis}

We now turn to the analysis of the spectrum of the family of ordinary differential operators, formally given by

\begin{equation}
    \Lren u = \partial_x^2 u + \sren \partial_x u + \fren'(\Pren) u. \label{eq:Lren}
\end{equation}

with $\fren$ and $\sren$ as defined in \eqref{eq:nonlinearity} and \eqref{eq:ren_quantities}, respectively. With a view towards the analysis of the stochastic Allen--Cahn Equation \eqref{eq:stochastic_AC_intro} on the full space $\bbR^2$ we must consider an integrable spatial weight for the underlying function space \cite{mourratGlobalWellposednessDynamic2017, gubinelliGlobalSolutionsElliptic2019}, which we chose as

\begin{equation}
    w_{\beta}(x) := \langle x \rangle^{- \beta} = (1 + \vert x \vert^2)^{-\beta/2}, \quad \beta > 2, \label{eq:weight}
\end{equation}

Its product with the same weight in the transversal direction $w_{\beta}(\bx) = w_{\beta}(x) w_{\beta}(x_{\perp})$ is integrable on $\bbR^2$ and equivalent (as a weight) to the one used in \cite{mourratGlobalWellposednessDynamic2017} to construct solutions which are global in time and space. Write $L^2(\bbR, w_{\beta})$ for the associated weighted $L^2$-space, which contains $L^2(\bbR)$ as a subspace and recall that the operator $T: L^2(\bbR, w_{\beta}) \rightarrow L^2(\bbR)$, defined by $Tu = w_{\beta}^{1/2} u$ is an isometric isomorphism. 

\begin{lem} \label{lem:Lren}
    Let $\varepsilon_0 > 0$ be as in Lemma \ref{lem:three_real_solutions} and let $\varepsilon \in B'(\varepsilon_0) \cap (0, \infty)$. Then the operator $\Lren$ with domain $H := T^{-1}(H^2(\bbR))$, as defined in \eqref{eq:Lren}, is densely defined and closed on $L^2(\bbR, w_{\beta})$, and $T$-conjugate to the closed linear operator $\Mren: H^2(\bbR) \subseteq L^2(\bbR) \rightarrow L^2(\bbR)$ defined by

    \begin{equation}
        \Mren := \partial_x^2 + \underbrace{\left( \sren + \frac{\beta x}{1+x^2} \right)}_{=: \aone(x)} \partial_x + \underbrace{\left( \sren \frac{\beta x}{2(1+x^2)} + \frac{(\beta^2 - 2 \beta)x^2 + 2 \beta}{4(1+x^2)^2} + \fren'(\Pren) \right)}_{=: \anull(x)}. \label{eq:Mren}
    \end{equation}
\end{lem}
\begin{proof}
    Via Lemma \ref{lem:conjugate_lin_op}, the dense definition and closedness of $\Lren$ follows from the corresponding properties of $\Mren$, the latter being closed due to $\aone, \anull \in L^{\infty}(\bbR)$ and Lemma \ref{lem:closed_lin_op}.   
    
    In order to compute the coefficients of $\Mren$, write $\phi(x) = \frac{1}{2} \log(w_{\beta}(x))$, so that $T u = e^{\phi} u$ and 
    
    \begin{equation}
        \phi'(x) = - \frac{\beta x}{2(1+x^2)}, \quad \phi''(x) = - \frac{\beta(1 - x^2)}{2(1+x^2)^2},
    \end{equation}
    
    and the commutator relations read

    \begin{equation}
        T \partial_x T^{-1} = \partial_x - \phi', \quad T \partial_x^2 T^{-1} = \partial_x^2 - 2 \phi' \partial_x - \phi'' + (\phi')^2.
    \end{equation}

    Plugging this into the definition of $\Mren$ yields

    \begin{align}
        \Mren &= T \Lren T^{-1} \\
        &= T \left( \partial_x^2 + \sren \partial_x + \fren'(\Pren) \right) T^{-1} \\
        &= T \partial_x^2 T^{-1} + \sren T \partial_x T^{-1} + \fren'(\Pren)\\
        &= \partial_x^2 - 2 \phi' \partial_x - \phi'' + (\phi')^2 + \sren \left( \partial_x - \phi' \right) + \fren'(\Pren) \\
        &= \partial_x^2 + \left( \sren - 2 \phi' \right) \partial_x + \left( (\phi')^2 - \phi'' - \sren \phi' + \fren'(\Pren) \right) \\
        &= \partial_x^2 + \left( \sren + \frac{\beta x}{1+x^2} \right) \partial_x + \left( \left( \frac{\beta x}{2(1+x^2)} \right)^2 + \frac{\beta (1 - x^2)}{2(1+x^2)^2} + \sren \frac{\beta x}{2(1+x^2)} + \fren'(\Pren) \right) \\
        &= \partial_x^2 + \left( \sren + \frac{\beta x}{1+x^2} \right) \partial_x + \left( \sren \frac{\beta x}{2(1+x^2)} + \frac{(\beta^2 - 2 \beta)x^2 + 2 \beta}{4(1+x^2)^2} + \fren'(\Pren) \right),
    \end{align}
    
    which finishes the proof.
\end{proof}

\paragraph{Decomposition of the spectrum}

We follow the definition of \cite[Ch. IV, §5]{katoPerturbationTheoryLinear1995}, according to which a closed linear operator $\sL$ on a Banach space $X$ is Fredholm if its range is closed and both its kernel and cokernel are finite-dimensional. We denote the Fredholm index by $\ind(\cdot)$ and split the spectrum of $\sL$ into the following three parts: 

\begin{align}
    \sigma_{\not F}(\sL) &:= \{ \lambda \in \bbC : \sL - \lambda I \text{ is not Fredholm}\}, \label{eq:sigma_not_F}\\
    \sigma_{k}(\sL) &:= \{ \lambda \in \bbC : \sL - \lambda I \text{ is Fredholm and } \ind(\sL - \lambda I) \neq k \}, \quad k \geq 1, \label{eq:sigma_k} \\
    \sigma_{0}(\sL) &:= \{ \lambda \in \bbC : \sL - \lambda I \text{ is not invertible, Fredholm and } \ind(\sL - \lambda I) = 0 \}. \label{eq:sigma_0}
\end{align}

The three sets are clearly disjoint. To see that they comprise the entire spectrum of $\sL$, recall that invertibility of a closed linear operator implies it being Fredholm with index $0$. Multiple (in general) non-equivalent definitions of the essential spectrum exist in the literature - see \cite{gustafsonEssentialSpectrum1969, jeribiGustafsonWeidmannKato2011}. We follow the definition of what is called the \emph{Schechter} essential spectrum $\sigma_{e5}$ in \cite[Sec. 2]{jeribiGustafsonWeidmannKato2011},

\begin{equation}
    \sigma_{\ess}(\sL) := \sigma_{\not F}(\sL) \cup \bigcup_{k \geq 1} \sigma_{k}(\sL), \label{eq:sigma_ess}
\end{equation}

which is the standard choice in the analysis of travelling waves - see e.g. \cite[Def. 2.2.3]{kapitulaSpectralDynamicalStability2013}, \cite[Def. 3.2]{sandstedeChapter18Stability2002}. This definition has the advantage that 1) Weyl's theorem \cite[Thm. VI.5.35.]{katoPerturbationTheoryLinear1995} on the invariance of the essential spectrum under relatively compact perturbations holds true,\footnote{\cite[Thm. IV.5.26.]{katoPerturbationTheoryLinear1995} ensures that relatively compact perturbations preserve not only the Fredholm property, but also the Fredholm index of a closed linear operator.} and 2) with this definition $\sigma_{\ess}$ is large enough to ensure that the complement is contained in the set of eigenvalues. 


\subsection{Essential spectrum} \label{subsec:essential_spectrum}
The following section is concerned with Theorem \ref{thm:sigma_ess}, the proof of which will be presented on p.~\pageref{page:essential}, after establishing the necessary intermediate results.

\begin{thm} \label{thm:sigma_ess}
    Let $\Lren$ be the closed linear operator defined in \eqref{eq:Lren} and Lemma \ref{lem:Lren} on $L^2(\bbR, w_{\beta})$ with domain $H$. There exists $\varepsilon_0 > 0$ such that

    \begin{equation}
        \sigma_{\ess}(\Lren) \cap \{ z \in \bbC : \Real(z) \gtrsim \varepsilon^{-2} + \calO(\varepsilon^{-1}) \} = \emptyset, \quad \varepsilon \in B'(\varepsilon_0) \cap (0, \infty).\label{eq:sigma_ess_thm}
    \end{equation}
\end{thm}

As is common practice for this type of problem, we will analyse the essential spectrum in two steps. Firstly, we will decompose the operator $\Mren$ as the sum of a piecewise constant operator 

\begin{equation}
    \Minfty = \partial_x^2 + \aone^{\infty}(x) \partial_x + \anull^{\infty}(x) 
\end{equation}

with 

\begin{equation}
    \ak^{\infty}(x) := \begin{cases}
        \ak^+, & x < -1, \\
        \ak^-, & x > 1, \\
        \text{smooth} & \text{else},
    \end{cases} \qquad \ak^{\pm} := \lim_{x \rightarrow \pm \infty} \ak(x),
\end{equation}

for $k = 0,1$, and a relatively compact perturbation $\Kren := \Mren - \Minfty$. Via Weyl's theorem on the invariance of the essential spectrum \cite[Thm. IV.5.35.]{katoPerturbationTheoryLinear1995}, the essential spectrum of $\Mren$ coincides with that of $\Minfty$. Then, in a second step, we will study the essential spectrum of $\Minfty$ by analysing the associated ODEs on the half-lines $\bbR_- := (-\infty, 0]$ and $\bbR_+ := [0, \infty)$ and their exponential dichotomies - see also \cite[Ch. 3]{kapitulaSpectralDynamicalStability2013}, \cite[Ch. 3.3]{sandstedeChapter18Stability2002}. We deviate slightly from \cite[Ch. 3]{kapitulaSpectralDynamicalStability2013} in the fact that the coefficients of $\Mren$ are not \emph{exponentially} asymptotically constant, but only have algebraic decay to their limits as $x \rightarrow \pm \infty$. This is a consequence of the \emph{polynomial} weight we chose and prohibits the direct use of the results in \cite[Ch. 3]{kapitulaSpectralDynamicalStability2013}.

\begin{rem}
    The asymptotic results of this paper hold true (and the proofs are actually simpler) if we consider a smooth exponential weight 

    \begin{equation}
        \tilde{w}_{\beta}(x) = e^{- \beta \vert x \vert}, \quad \beta > 0, \vert x \vert > 1,
    \end{equation}

    instead of the polynomial weight defined in Equation \eqref{eq:weight}. This would produce coefficients $\aone, \anull$ which are constant outside of $[-1, 1]$. However, again with the view towards application to travelling wave problems in the stochastic Allen--Cahn equation on $\bbR^2$, we stick to the weights used in \cite{mourratGlobalWellposednessDynamic2017}.
\end{rem}

\subsubsection{Relatively compact perturbation}

\begin{lem} \label{lem:rel_compact_pert}
    Let $\varepsilon_0 > 0$ be as in Lemma \ref{lem:three_real_solutions} and let $\varepsilon \in B'(\varepsilon_0) \cap (0, \infty)$. The operator $\Kren := \Mren - \Minfty$, given by

    \begin{equation}
        \Kren = \left( \aone(x) - \aone^{\infty}(x) \right) \partial_x + \left( \anull(x) - \anull^{\infty}(x) \right),
    \end{equation}

    is a relatively compact perturbation of the closed linear operator $\Minfty$ on $L^2(\bbR)$ with domain $H^2(\bbR)$ as defined in Lemma \ref{lem:Lren}.
\end{lem}
\begin{proof}
    The fact that $\Minfty$ with domain $H^2(\bbR))$ is a closed linear operator on $L^2(\bbR)$ follows from Lemma \ref{lem:closed_lin_op}, which also shows that the resolvent set of $\Minfty$ is non-empty. To show compactness of $\Kren$ relative to $\Minfty$ it is thus enough to show that $\Kren (\Minfty - \lambda I)^{-1}: L^2(\bbR) \rightarrow L^2(\bbR)$ is compact for some $\lambda \in \rho(\Minfty)$. We factor the operator as

    \begin{equation}
        L^2(\bbR) \xrightarrow{(\Minfty - \lambda I)^{-1}} H^2(\bbR) \xrightarrow{\iota} H^1(\bbR) \xrightarrow{\Kren} L^2(\bbR), \label{eq:factoring}
    \end{equation}

    where $\iota$ denotes the canonical embedding. Because $\ak - \ak^{\infty} \in L^{\infty}(\bbR)$ with $\lim_{R \rightarrow \infty} \sup_{\vert x \vert > R} \vert \ak(x) - \ak^{\infty}(x) \vert \rightarrow 0$, Lemma \ref{lem:compact_multiplication} shows that $\Kren: H^1(\bbR) \rightarrow L^2(\bbR)$ is compact. Since the first two operators in \eqref{eq:factoring} are bounded and the last is compact, the product of the three is compact, which concludes the proof.
\end{proof}

\subsubsection{Fredholm borders and essential spectrum}

Let $A \in \bbC^{n \times n}$ be hyperbolic; i.e. $\sigma(A) \cap \ii \bbR = \emptyset$. Then define its Morse index $i(A)$ as the sum of the geometric multiplicities of the eigenvalues of $A$ with positive real part; i.e. the dimension of its generalized unstable subspace.

\begin{lem} \label{lem:Fred_ODE}
    Let $\varepsilon_0 > 0$ be as in Lemma \ref{lem:three_real_solutions}, let $\varepsilon \in B'(\varepsilon_0) \cap (0, \infty)$, and for any $\lambda \in \bbC$ define the matrices

    \begin{equation}
        \Asymppm(\lambda) := \begin{pmatrix}
            0 & 1 \\
            \lambda - \anull^{\pm}& - \aone^{\pm}
        \end{pmatrix}.
    \end{equation}
    
    Then for every $k \geq 0$,

    \begin{enumerate}
        \item $\lambda \in \sigma_{\not F}(\Minfty)$ if and only if $\Asympminus(\lambda)$ or $\Asympplus(\lambda)$ are not hyperbolic. 
        \item $\lambda \in \sigma_{k}(\Minfty)$ if and only if $\Asympminus(\lambda)$ and $\Asympplus(\lambda)$ are hyperbolic and $i(\Asympminus(\lambda)) - i(\Asympplus(\lambda)) = k$.
    \end{enumerate}
\end{lem}
\begin{proof}
    Since the coefficients of $\Minfty$ are smooth and real-valued (cf. \cite[Hyp. 3.1]{sandstedeChapter18Stability2002}), a given $\lambda \in \bbC$ satisfies $\lambda \in \sigma_{\not F}(\Minfty)$ if and only if the $2$-dimensional ODE 

    \begin{equation}
        \begin{pmatrix}
            u \\ u'
        \end{pmatrix}' = \begin{pmatrix}
            0 & 1 \\
            \lambda - \anull^{\infty}(x) & - \aone^{\infty}(x)
        \end{pmatrix} \begin{pmatrix}
            u \\ u'
        \end{pmatrix}, \label{eq:ODE}
    \end{equation}

    does not have an exponential dichotomy on $\bbR_-$ or on $\bbR_+$ - see \cite[Thm. 3.2 \& Rem. 3.3]{sandstedeChapter18Stability2002} and also \cite{palmerExponentialDichotomiesTransversal1984, palmerExponentialDichotomiesFredholm1988}. Given that on $\bbR_-$ and $\bbR_+$, respectively, $\ak^{\infty}(x) = \ak^{\pm}$ outside of a compact interval, \cite[Thm. 3.1]{sandstedeChapter18Stability2002} shows that the existence of an exponential dichotomy of \eqref{eq:ODE} on each interval is equivalent to $\Asympminus(\lambda)$ or $\Asympplus(\lambda)$ not being hyperbolic. This shows the first claim.

    Similarly, \cite[Thm. 3.2 \& Rem. 3.3]{sandstedeChapter18Stability2002} shows that $\lambda \in \sigma_k(\Minfty)$ if and only if the ODE \eqref{eq:ODE} has an exponential dichotomy on both $\bbR_-$ and $\bbR_+$ and the difference of the dimension of the associated unstable subspaces equals $k$. But the same argument as above this is equivalent to the Morse indices of $\Asymppm(\lambda)$ satisfying $i(\Asympminus(\lambda)) - i(\Asympplus(\lambda)) = k$. This shows the second claim.
\end{proof}

\begin{rem}
    The operator $\Minfty$ is what is called the asymptotic operator in \cite[Ch. 3]{kapitulaSpectralDynamicalStability2013}. However, we keep the coefficients smooth in order to be able to apply the results in \cite{sandstedeChapter18Stability2002} directly (cf. \cite[Hyp. 3.1]{sandstedeChapter18Stability2002}). The matrices $\Asymppm(\lambda)$ correspond exactly to what is called the asymptotic matrices associated to the asymptotic operator in \cite[Ch. 3]{kapitulaSpectralDynamicalStability2013}.
\end{rem}

\begin{lem} \label{lem:ess_spectrum_M_infty}
    Let $\varepsilon_0 > 0$ be as in Lemma \ref{lem:three_real_solutions} and let $\varepsilon \in B'(\varepsilon_0) \cap (0, \infty)$. Then

    \begin{equation}
        \sigma_{\ess}(\Minfty) \subseteq \left\{ \lambda \in \bbC : \Real(\lambda) \leq - 6 \varepsilon^{-2} + \calO(\varepsilon^{-1}) \right\}.
    \end{equation}
\end{lem}
\begin{proof}
    Firstly, by Lemma \ref{lem:Fred_ODE}, the set $\sigma_{\not F}(\Minfty)$ coincides with 

    \begin{equation}
         \{ \lambda \in \bbC : \Asympminus(\lambda) \text{ or } \Asympplus(\lambda) \text{ is not hyperbolic} \}, \label{eq:fredholm_border}
    \end{equation}

    which in turn may be characterized as the set of those $\lambda \in \bbC$ for which the characteristic polynomial of $\Asympminus(\lambda)$ or of $\Asympplus(\lambda)$ has a purely imaginary root; i.e. 

    \begin{equation}
        \det(\Asymppm(\lambda) - (\ii \xi) I) = 0, \quad \text{for some } \xi \in \bbR.
    \end{equation}

    We may thus parametrize $\sigma_{\not F}(\Minfty)$ by the two curves
    
    \begin{equation}
        \Gamma^{\pm}: \bbR \rightarrow \bbC; \quad \xi \mapsto (\ii \xi)^2 + \aone^{\pm} \ii \xi + \anull^{\pm}.
    \end{equation}

    In particular, the point of $\sigma_{\not F}(\Minfty)$ with largest real part is attained at $\xi = 0$ and given by $\max \{\anull^-, \anull^+\}$.

    Secondly, by \cite[Ch. IV, §5]{katoPerturbationTheoryLinear1995}, the set 

    \begin{equation}
        \bigcup_{k \geq 0} \sigma_k(\Minfty) \label{eq:union}
    \end{equation}

    is the disjoint union of a countable number of open and connected sets on which $\lambda \mapsto \ind (\Minfty - \lambda I)$ is constant, and the boundary of each is contained in $\sigma_{\not F}(\Minfty)$. Therefore, by the first point above, \eqref{eq:union} is comprised of at most three such sets - see Figure \ref{fig:essential_spectrum}: left (cross-hatched), middle (linearly hatched), right (blank). In order to rule out that the right (blank) region is part of $\sigma_{\ess}(\Minfty)$, note that for $\lambda$ real and large enough we have 

    \begin{equation}
        \det(\Asymppm(\lambda)) = \anull^{\pm} - \lambda < 0. \label{eq:det_negative}
    \end{equation}

    Since $\lambda \in \bbR$, the matrices $\Asymppm(\lambda)$ are of dimension $2 \times 2$ with real entries, and hence Equation \eqref{eq:det_negative} implies that both $\Asymppm(\lambda)$ have one positive and one negative eigenvalue; i.e. $i(A^-(\lambda)) = i(A^+(\lambda)) = 1$. Thus, by Lemma \ref{lem:Fred_ODE}, $\lambda \notin \sigma_{\ess}(\Minfty)$ for some $\lambda$ real and large enough and thus for the entire right (blank) region. Therefore,

    \begin{equation}
        \sigma_{\ess}(\Minfty) \subseteq \left\{ \lambda \in \bbC : \Real(\lambda) \leq \max \{\anull^-, \anull^+\} \right\}.
    \end{equation}
    
    For the asymptotic expansion, notice firstly that the only summand in the definition of $\anull^{\pm}$ that does not vanish at $\pm \infty$ is $\fren'(\Pren(x))$, and thus
    
    \begin{equation}
        \anull^{\pm} = \lim_{x \rightarrow \pm \infty} \fren'(\Pren(x)). \label{eq:asymptotic_a0}
    \end{equation}

    Combining $\lim_{x \rightarrow \pm \infty} \Pren(x) = \zpm$ from Theorem \ref{thm:wave_data} with the continuity of $\fren'$ to pull the limit inside, we obtain

    \begin{equation}
        \anull^{\pm} = \fren'(\zpm) = - 3 (\zpm)^2 + 2 (1 + \alpha) \zpm + 3 \varepsilon^{-2} - \alpha. \label{eq:a0_in_terms_of_z}
    \end{equation}

    The expansion now follows from direct substitution of the Laurent series from Lemma \ref{lem:three_real_solutions}.
\end{proof}

\begin{figure}[htbp]
    \centering
    \includegraphics[width=0.6\textwidth]{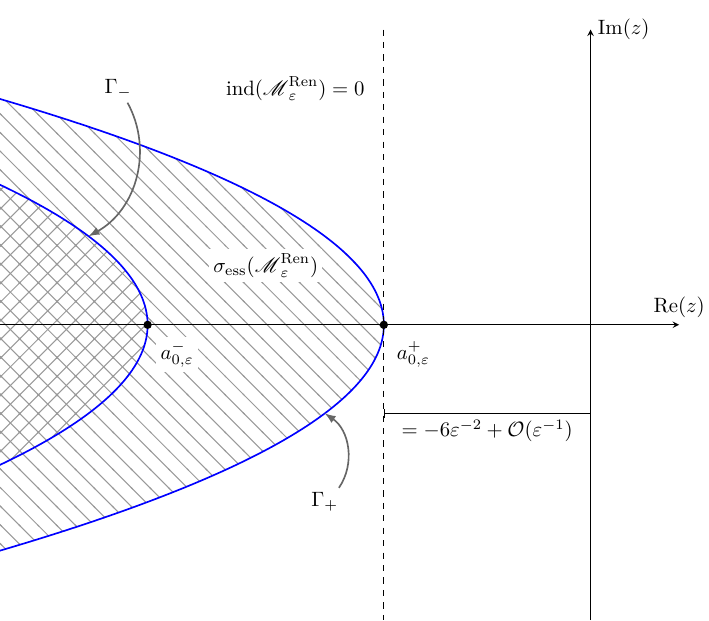} 
    \caption{Essential spectrum of $\Minfty$.}
    \label{fig:essential_spectrum}
\end{figure}

\begin{proof}[Proof of Theorem \ref{thm:sigma_ess}] 
\label{page:essential}
    By Lemma \ref{lem:conjugate_lin_op}, conjugation preserves the decomposition of the spectrum into \eqref{eq:sigma_not_F}, \eqref{eq:sigma_k}, and \eqref{eq:sigma_0}, and hence $\sigma_{\ess}(\Lren) = \sigma_{\ess}(\Mren)$. By Weyl's theorem \cite[Thm. IV.5.35.]{katoPerturbationTheoryLinear1995} and Lemma \ref{lem:rel_compact_pert}, the latter coincides with $\sigma_{\ess}(\Minfty)$, which satisfies \eqref{eq:sigma_ess_thm} by Lemma \ref{lem:ess_spectrum_M_infty}.
\end{proof}


\subsection{Point spectrum} \label{subsec:discrete_spectrum}

The following section is concerned with Theorem \ref{thm:sigma_0}, the proof of which will be presented on p.~\pageref{page:point}, after establishing the necessary intermediate results.

\begin{thm} \label{thm:sigma_0}
   Let $\Lren$ be the closed linear operator defined in \eqref{eq:Lren} and Lemma \ref{lem:Lren} on $L^2(\bbR, w_{\beta})$ with domain $H$. There exists $\varepsilon_0 > 0$ such that

    \begin{equation}
        \sigma_0(\Lren) \cap \{ z \in \bbC : \Real(z) \gtrsim \varepsilon^{-2} + \calO(\varepsilon^{-1}) \} = \{ 0 \}, \quad \forall \varepsilon \in B'(\varepsilon_0) \cap (0, \infty). \label{eq:sigma_0_statement}
    \end{equation}
\end{thm}

\subsubsection{Reduction to a Schrödinger eigenvalue problem}

We start by introducing the relevant family of Schrödinger operators.

\begin{prop} \label{prop:Schrödinger}
    Let $\varepsilon_0 > 0$ be as in Lemma \ref{lem:three_real_solutions} and let $\varepsilon \in B'(\varepsilon_0) \cap (0, \infty)$ be arbitrary. Then the operator $\Hren: L^2(\bbR) \supseteq H^2(\bbR) \rightarrow L^2(\bbR); u \mapsto - \partial_x^2 u + \Qren u$, with 

    \begin{equation}
        \Qren(x) :=  - \fren'(\Pren(x)) + \frac{1}{4} \sren^2. \label{eq:Q_ren}
    \end{equation}

    is self-adjoint with essential spectrum satisfying $\sigma_{\ess}(\Hren) \subseteq [6 \varepsilon^{-2} + \calO(\varepsilon^{-1}), \infty)$.
\end{prop}
\begin{proof}
    Akin to Section \ref{subsec:essential_spectrum}, we can decompose $\Qren$ into the sum of
    
    \begin{equation}
    \Qinf(x) := \begin{cases}
        Q^{-}_{\varepsilon}, & x < -1, \\
        Q^{+}_{\varepsilon}, & x > 1, \\
        \text{smooth and monotone} & \text{else},
    \end{cases}, \qquad Q_{\varepsilon}^{\infty, \pm} := \lim_{x \rightarrow \pm \infty} \Qren(x) = - \anull^{\pm},
    \end{equation}

    and $K \in L^{\infty}(\bbR)$ such that $\lim_{R \rightarrow \infty} \sup_{\vert x \vert > R} \vert K(x) \vert = 0$. Since $\Qren$ and $\Qinf$ are bounded and real, their associated multiplication operators are bounded and symmetric. Hence, via the Kato--Rellich theorem \cite[Thm. V.4.11.]{katoPerturbationTheoryLinear1995}, both $- \partial_x^2 + \Qinf$ and $- \partial_x^2 + \Qren$ are self-adjoint on $L^2(\bbR)$ with domain $H^2(\bbR)$ and bounded from below. Furthermore, by the same argument as in Lemma \ref{lem:rel_compact_pert} (using compactness of the $K$-multiplication operator relative to $-\Delta$ and Lemma \ref{lem:compact_multiplication}), Weyl's theorem \cite[Thm. IV.5.35.]{katoPerturbationTheoryLinear1995} shows that their essential spectra agree. Thus, finally, since $\Qinf$ is asymptotically constant and monotone, $\inf_{x \in \bbR} \Qinf(x) = \min \{ Q_{\varepsilon}^{\infty, -}, Q_{\varepsilon}^{\infty, +} \}$ and the variational inequality \cite[Prop. 12.1]{hislopIntroductionSpectralTheory1996} shows that

    \begin{equation}
        \inf \sigma_{\ess}(- \partial_x^2 + \Qren) = \inf \sigma_{\ess}(- \partial_x^2 + \Qinf) \geq \inf_{u \in H^2(\bbR)} \langle (- \partial_x^2 + \Qinf) u, u \rangle \geq \min \{ Q_{\varepsilon}^{\infty, -}, Q_{\varepsilon}^{\infty, +} \}. \label{eq:variational_lower_bound}
    \end{equation}

    Since

    \begin{equation}
        Q_{\varepsilon}^{\infty, \pm} = \underbrace{- \fren'(\zpm)}_{= 6 \varepsilon^{-2} + \calO(\varepsilon^{-1})} + \frac{1}{4} \underbrace{\sren^2}_{\calO(\varepsilon^{4})},
    \end{equation}

    the estimate \eqref{eq:variational_lower_bound} shows the inclusion for the essential spectrum.
\end{proof}

\begin{figure}
    \centering
    \includegraphics[width=0.6\textwidth]{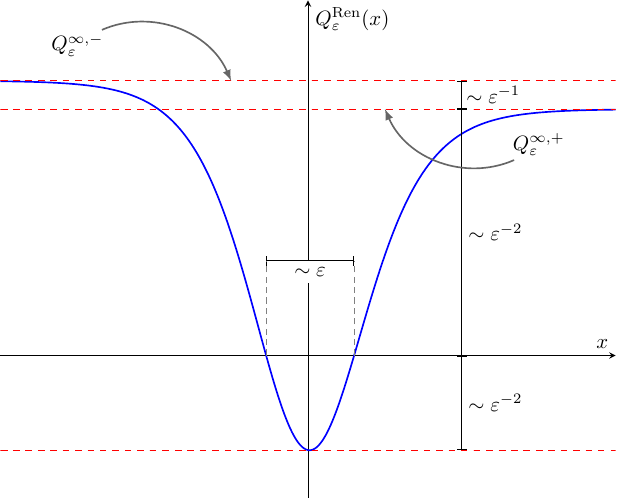} 
    \caption{Plot of the potential $\Qren$ from \eqref{eq:Q_ren} for a fixed $\varepsilon \in B'(\varepsilon_0) \cap (0, \infty)$.}
    \label{fig:potential_Q_ren}
\end{figure}

The general idea is to conjugate $\Lren$ to the Schrödinger operator $\Hren$, defined in Proposition \ref{prop:Schrödinger}, via an appropriate multiplication operator 

\begin{equation}
    S: L^2(\bbR) \rightarrow L^2(\bbR, w_{\beta}); \quad u \mapsto \exp\left( - \frac{1}{2} \sren x \right) u,
\end{equation}

that eliminates the first order coefficient; i.e. $\Hren := S^{-1} \Lren S = - \partial_x^2 + \Qren$ for an appropriate potential $\Qren \in L^{\infty}$. However, the multiplier $S$, which is fixed by the requirement of eliminating $\sren \partial_x$, is not a bounded (or even well-defined) linear operator, and even if it were, the domain of $\Hren$ would not be $H^2(\bbR)$, and thus $\Hren$ not self-adjoint. 

However, since we are only interested in a portion of the spectrum of $\Lren$, and the involved operators make sense on spaces with much weaker integrability constraints, we can utilize Agmon decay; that is, eigenfunctions of Schrödinger operators decay very fast outside of classically allowed regions (which is the complement of a compact set in our case). This allows to show "by hand" that eigenvalues of $\Lren$ that lie outside of the essential spectrum are also eigenvalues of $\Hren$. 

\begin{prop} \label{prop:eigenvector_of_Schrödinger}
    For $\varepsilon_0 > 0$ as in Lemma \ref{lem:three_real_solutions} and $\varepsilon \in B'(\varepsilon_0) \cap (0, \infty)$, define

    \begin{equation}
        \rho_{\varepsilon}(x) := \exp\left( - \frac{1}{2} \sren x \right), \quad x \in \bbR. \label{eq:rho_eps}
    \end{equation}
    
    Let $\lambda \in \sigma_0(\Lren)$ with $\Real(-\lambda) < \inf \sigma_{\ess}(\Hren)$ and associated eigenvector $u \in \ker( \Lren - \lambda I)$. Then there exists $\varepsilon_0 > 0$ small enough such that for all $\varepsilon \in B'(\varepsilon_0) \cap (0, \infty)$,

    \begin{equation}
        \rho_{\varepsilon} u \in H^2(\bbR), \quad \text{and} \quad \Hren (\rho_{\varepsilon} u) = - \lambda (\rho_{\varepsilon} u) \label{eq:Schrodinger_eigenfunction}
    \end{equation}
\end{prop}

\textbf{Attention:} Note the switch of signs; i.e. $\lambda \in \sigma_{0}(\Lren)$ implies $-\lambda \in \sigma_0(\Hren)$. This results from the choice of making $\sH$ positive (instead of negative) which is more in line with the Schrödinger operator literature and intuition. 

\begin{proof}[Proof of Proposition \ref{prop:eigenvector_of_Schrödinger}]
    Let $u \in H$ such that $\Lren u = \lambda u$. Then, in the distributional sense, via

    \begin{equation}
        \rho_{\varepsilon}' = - \frac{\sren}{2} \rho_{\varepsilon}, \quad (\rho_{\varepsilon} u)' = - \frac{\sren}{2} \rho_{\varepsilon} u + \rho_{\varepsilon} \partial_x u \quad (\rho_{\varepsilon} u)'' = \frac{\sren^2}{4} \rho_{\varepsilon} u - \sren \rho_{\varepsilon} \partial_x u + \rho_{\varepsilon} \partial_x^2 u
    \end{equation}

    we have 
    
    \begin{align}
        - \Hren (\rho_{\varepsilon} u) &= (\rho_{\varepsilon} u)'' - \Qren (\rho_{\varepsilon} u) \\
        &= \frac{\sren^2}{4} \rho_{\varepsilon} u \underbrace{- \sren \rho_{\varepsilon} \partial_x u + \rho_{\varepsilon} \partial_x^2 u + \fren'(\Pren)(x) \rho_{\varepsilon} u}_{= \rho_{\varepsilon} \Lren u = \rho_{\varepsilon} \lambda u} - \frac{1}{4} \sren^2 \rho_{\varepsilon} u \\
        &= \rho_{\varepsilon} \lambda u = \lambda (\rho_{\varepsilon} u).
    \end{align}

    It is thus only left to show that the distribution $\rho_{\varepsilon} u$ lies in $H^2(\bbR)$. This is only a question of decay (not regularity), and so we want to apply the decay result of \cite[Thm. 4.1]{agmonLecturesExponentialDecay2014}.

    Since $\Qren$ is real and bounded for $\varepsilon \in B'(\varepsilon_0) \cap (0, \infty)$, the operator $- \partial_x^2 + \Qren$ with domain $H^2(\bbR)$ is semi-bounded on $L^2(\bbR)$ by \cite[Thm. V.4.11.]{katoPerturbationTheoryLinear1995}, and therefore the operator $\Hren$ is of the form required in \cite[Thm. 4.1]{agmonLecturesExponentialDecay2014}. In addition, by \cite[Thm 3.2]{agmonLecturesExponentialDecay2014}\footnote{Since $\Hren$ with domain $H^2(\bbR)$ is self-adjoint on $L^2(\bbR)$, the definition of the essential spectrum in \cite[Thm 3.2]{agmonLecturesExponentialDecay2014} coincides with our definition in \eqref{eq:sigma_ess} - see \cite[Sec. 2]{jeribiGustafsonWeidmannKato2011}.} and by assumption

    \begin{align}
        \Sigma(- \partial_x^2 + \Qren) &:= \sup_{K \Subset \bbR} \inf \left\{ \frac{\langle (- \partial_x^2 + \Qren) \varphi, \varphi \rangle}{\langle \varphi, \varphi \rangle} : \varphi \in C^{\infty}_c(\bbR \setminus K), \varphi \neq 0 \right\} = \inf \sigma_{\ess}(\Hren) > \Real(- \lambda). \label{eq:bounded_from_below}
    \end{align}

    Therefore, we may choose $\gamma$ such that $0 < \gamma < (\Sigma + \Real(\lambda))^{1/2}$, and because $\sren = \calO(\varepsilon^2)$ (see Theorem \ref{thm:wave_data}), we may shrink $\varepsilon_0 > 0$ such that $\vert \sren \vert - 2 \gamma < 0$ for all $\varepsilon \in B'(\varepsilon_0) \cap (0, \infty)$. Then, writing $u = w_{\beta}^{-1/2} v$ with $v \in H^2(\bbR)$, the weak integrability condition

    \begin{equation}
        \int_{\bbR} \vert \rho_{\varepsilon} u \vert^2 \exp(-2 \gamma \vert x \vert) \dd x = \int_{\bbR} \vert v \vert^2 w_{\beta}^{-1}\exp( \underbrace{(\vert \sren \vert - 2 \gamma)}_{< 0} \vert x \vert) \dd x < \infty \label{eq:weak_int}
    \end{equation}
    
    is satisfied and thus by \cite[Thm. 4.1]{agmonLecturesExponentialDecay2014} we obtain the strong integrability condition

    \begin{equation}
        \int_{\bbR} \vert \rho_{\varepsilon} u \vert^2 \exp(2 \gamma \vert x \vert) \dd x < \infty,
    \end{equation}

    and in particular $\rho_{\varepsilon} u \in H^2(\bbR)$.
\end{proof}

\begin{rem}
    There are weaker and more standard versions of Agmon's decay theorem than the one used here, e.g. \cite[Thm. 3.4]{hislopIntroductionSpectralTheory1996}. But since $H \not\subseteq L^2(\bbR)$, eigenfunctions of $\Lren$ are not necessarily square integrable, and hence we cannot apply them directly.
\end{rem}

\subsubsection{Meromorphic perturbation}

Plugging $\Pren$ into $\fren'$ shows that for $\varepsilon$ real and small enough, the depth of the well of $\Qren$ grows like $\varepsilon^{-2}$, while its width shrinks like $\varepsilon$ - see Figure \ref{fig:potential_Q_ren}. Hence, we expect that 

\begin{equation}
    \Qren(x) = \varepsilon^{-2} \Qhol(\varepsilon^{-1} x), \quad \varepsilon \in B'(\varepsilon_0) \cap (0, \infty), x \in \bbR, \label{eq:impl_Qhol}
\end{equation}

for some family of potentials $\Qhol$ that remain bounded as $\varepsilon \rightarrow 0$. From scaling heuristics for Schrödinger operators, one can then already conjecture that $\sigma(\Hren) = \varepsilon^{-2} \sigma(\Hhol)$ for the corresponding family of operators $\Hhol := - \partial_x^2 + \Qhol$ which will be obtained from $\Hren$ via dilation operators. Since all the data considered up to this point was holomorphic/meromorphic in $\varepsilon$, we expect that the family $\{ \Hhol : \varepsilon \in B(\varepsilon_0) \}$ is a holomorphic family in the sense of Kato, which will allow to apply analytic perturbation theory and close the argument.

For $\varepsilon > 0$, let 

\begin{equation}
    \dila: L^2(\bbR) \rightarrow L^2(\bbR); \quad u \mapsto \varepsilon^{1/2} u(\varepsilon \cdot)
\end{equation}

denote the real dilation operators and recall that $\dila$ is unitary with inverse $\dila^{-1} = U_{\varepsilon^{-1}}$.

\begin{lem} \label{lem:scaling_Hren_Hhol}
    Let $\varepsilon_0 > 0$ be as in Lemma \ref{lem:three_real_solutions} and let $\varepsilon \in B'(\varepsilon_0) \cap (0, \infty)$. Then the operator $\Hhol := - \partial_x^2 u + \Qhol$ with
    
    \begin{equation}
        \Qhol(x) := \varepsilon^2 \Qren(\varepsilon x), \quad x \in \bbR, \label{eq:Q_hol}
    \end{equation}

    is self-adjoint on $H^2(\bbR) \subseteq L^2(\bbR)$ and satisfies  

    \begin{equation}
        \dila \Hren \dila^{-1} = \varepsilon^{-2} \Hhol.
    \end{equation}
    
    In particular, 

    \begin{equation}
        \varepsilon^2 \sigma(\Hren) = \sigma(\Hhol), \quad \varepsilon \in B'(\varepsilon_0) \cap (0, \infty).
    \end{equation}
\end{lem}
\begin{proof}
    For $\varepsilon \in B'(\varepsilon_0) \cap (0, \infty)$, the function $\Qren$ is real-valued, bounded, and symmetric, and therefore so is $\Qhol$. Thus, by the Kato--Rellich theorem \cite[Thm. V.4.11.]{katoPerturbationTheoryLinear1995}, $\Hhol$ is self-adjoint on $H^2(\bbR)$. Let $u \in H^2(\bbR)$ be arbitrary. Then 

    \begin{align}
        \dila \Hren \dila^{-1} u &= \varepsilon^{-1/2} \dila \left( - \partial_x^2 (u(\varepsilon^{-1} \cdot)) - \Qren u(\varepsilon^{-1} \cdot) \right) \\
        &= - \varepsilon^{-2} \Delta u - \Qren(\varepsilon \cdot) u \\
        &= \varepsilon^{-2}( - \partial_x^2 - \underbrace{\varepsilon^2 \Qren(\varepsilon \cdot)}_{= \Qhol}) u \\
        &= \varepsilon^{-2} \Hhol u.
    \end{align}
    This finishes the proof.
\end{proof}

\begin{prop} \label{prop:holo_family}
    Let $\varepsilon_0 > 0$ be as in Lemma \ref{lem:three_real_solutions}. Then the family of self-adjoint linear operators $\{ - \partial_x^2 + \Qhol : \varepsilon \in B'(\varepsilon_0) \cap (0, \infty)\}$ has an extension $\{ - \partial_x^2 + \Qhol : \varepsilon \in B(\varepsilon_0)\}$ that is a holomorphic family in the sense of Kato.
\end{prop}

We postpone the proof of Proposition \ref{prop:holo_family} until after Lemma \ref{lem:Phi_holo} and Lemma \ref{lem:Q_holo}.

\begin{lem} \label{lem:Phi_holo}
    There exists $\varepsilon_0 > 0$ small enough such that the function $\Pholnoeps: B'(\varepsilon_0) \cap (0, \infty) \rightarrow L^{\infty}(\bbR)$, defined by

    \begin{equation}
        \Phol(x) := \varepsilon \Pren(\varepsilon x) \quad x \in \bbR, \label{eq:Phol}
    \end{equation}

    has a holomorphic extension to all of $B(\varepsilon_0)$, satisfying 
    
    \begin{equation}
        \sup_{x \in \bbR} \left\vert \Phol(x) - \Phollim(x) - \varepsilon \frac{1 + \alpha}{3} \right\vert \lesssim \vert \varepsilon \vert^{3/2}. \label{eq:Phol_estimate}
    \end{equation}
\end{lem}
\begin{proof}
    Define an extension of $\Phol$ to $B(\varepsilon_0)$ by 

    \begin{equation}
        \Phollim(x) = \lim_{\varepsilon \rightarrow 0} \varepsilon \Pren(\varepsilon x) = \sqrt{12} \Phux(\sqrt{12} x) - \sqrt{3}. \label{eq:Phollim}
    \end{equation}

    for $\varepsilon = 0$ and by the formula \eqref{eq:Phol} for $\varepsilon \in B'(\varepsilon_0)$ (recall that $\varepsilon \mapsto \zminus, \zzero, \zplus$ is defined for all $\varepsilon \in B'(\varepsilon_0)$). Holomorphicity with values in $L^{\infty}(\bbR)$ then follows from \eqref{eq:Phol_estimate}. 
    
    To see \eqref{eq:Phol_estimate} fix $\varepsilon \in B(\varepsilon_0)$. Firstly, via \eqref{eq:Phol} and \eqref{eq:ren_quantities} we have 

    \begin{align}
        &\left\vert \Phol(x) - \Phollim(x) - \varepsilon \frac{1 + \alpha}{3} \right\vert \\
        &\leq \left\vert \left( \varepsilon \Aren \Phux(\varepsilon \Aren x) + \varepsilon \zminus \right) - \left( \sqrt{12} \Phux(\sqrt{12} x) - \sqrt{3} \right) - \varepsilon \frac{1 + \alpha}{3} \right\vert \\
        &\lesssim \underbrace{\left\vert \Phux(\sqrt{12} x) \right\vert}_{\lesssim 1} \underbrace{\left\vert \varepsilon \Aren - \sqrt{12} \right\vert}_{\lesssim \vert \varepsilon \vert^2, ~ \text{Thm.} \ref{thm:wave_data}} + \underbrace{\left\vert \Phux (\varepsilon \Aren x) - \Phux(\sqrt{12} x) \right\vert}_{=: I_{\varepsilon}} \underbrace{\vert \varepsilon \Aren \vert}_{\lesssim 1} + \underbrace{\left\vert \varepsilon \zminus + \sqrt{3} - \varepsilon \frac{1 + \alpha}{3} \right\vert}_{\lesssim \vert \varepsilon \vert^2, ~ \text{Lem.} \ref{lem:three_real_solutions}}.
    \end{align}
    
    It is thus sufficient to show that $I_{\varepsilon} \lesssim \varepsilon^{3/2}$ independently of $x \in \bbR$. Define the following three regions of $\bbC$ - see Figure \ref{fig:sigma_regions}:

    \begin{align}
        \Sigma^+ &:= \{ z \in \bbC : \vert \Ima(z) \vert \leq \vert \Real(z) \vert, \Real(z) \geq 0 \}, \label{eq:sigma_plus} \\
        \Sigma^- &:= \{ z \in \bbC : \vert \Ima(z) \vert \leq \vert \Real(z) \vert, \Real(z) < 0 \}, \label{eq:sigma_minus} \\
        \Sigma &:= \Sigma^- \cup \Sigma^+ \label{eq:sigma_full}
    \end{align}
    
    Recall from Theorem \ref{thm:wave_data} that 

    \begin{equation}
        \varepsilon \Aren = \sqrt{12} + \calO(\varepsilon^2), \label{eq:Aren_asymptotic}
    \end{equation}

    i.e. that $\varepsilon \Aren$ lies in a ball of radius $\calO(\vert \varepsilon \vert^2)$ around $\sqrt{12} \in \Sigma^+$. By potentially shrinking $\varepsilon_0 > 0$ even smaller, we may thus assume that $\varepsilon \Aren \in \Sigma^+$, implying in turn that $\{ \varepsilon \Aren x : x \in \bbR_{\pm} \} \subseteq \Sigma^{\pm}$. We then distinguish three cases:

    \begin{enumerate}   
        \item Let $\vert x \vert \leq \vert\varepsilon\vert^{-1/2}$. Since both $\varepsilon \Aren x$ and $\sqrt{12} x$ lie in the same semi-cone $\Sigma^{\pm}$, so does the line segment connecting the two. Thus, since $\Phux$ is holomorphic on a neighborhood of $\Sigma$ by Lemma \ref{lem:exp_decay}, we have
         
        \begin{align}
            \vert \Phux(\varepsilon \Aren x) - \Phux(\sqrt{12}x) \vert &\leq \underbrace{\sup_{z \in \Sigma} \vert \Phux'(z) \vert}_{\lesssim 1, ~ \text{Lem.} \ref{lem:exp_decay}} \underbrace{\vert \varepsilon \Aren - \sqrt{12} \vert}_{\lesssim \vert\varepsilon\vert^2} \underbrace{\vert x \vert}_{\lesssim \vert\varepsilon\vert^{-1/2}} \lesssim \vert\varepsilon\vert^{3/2}.
        \end{align}

    \item Let $x > \vert\varepsilon\vert^{-1/2}$. Then $\varepsilon \Aren x \in \Sigma^+$, and via the decay estimate in $\Sigma^+$ from Lemma \ref{lem:exp_decay}, we have

        \begin{align}
            \vert \Phux(\varepsilon \Aren x) - \Phux(\sqrt{12}x) \vert &\leq \vert \Phux(\varepsilon \Aren x) - 1 \vert + \vert \Phux(\sqrt{12}x) - 1 \vert \\
            &\lesssim \exp\Bigg(- \underbrace{\frac{\Real(\varepsilon \Aren)}{\sqrt{2}}}_{\gtrsim 1} \underbrace{x}_{> \vert\varepsilon\vert^{-1/2}} \Bigg) + \exp\Bigg(- \underbrace{\frac{\sqrt{12}}{\sqrt{2}}}_{\gtrsim 1} \underbrace{x}_{> \vert\varepsilon\vert^{-1/2}} \Bigg) \\
            &\lesssim \vert\varepsilon\vert^{3/2}.
        \end{align}

    \item Let $x < -\vert\varepsilon\vert^{-1/2}$. Then $\varepsilon \Aren x \in \Sigma^-$, and via the decay estimate in $\Sigma^-$ from Lemma \ref{lem:exp_decay}, we have

        \begin{align}
            \vert \Phux(\varepsilon \Aren x) - \Phux(\sqrt{12}x) \vert &\leq \vert \Phux(\varepsilon \Aren x) \vert + \vert \Phux(\sqrt{12}x) \vert \\
            &\lesssim \exp\Bigg(\underbrace{\frac{\Real(\varepsilon \Aren)}{\sqrt{2}}}_{\lesssim 1} \underbrace{x}_{< -\vert\varepsilon\vert^{-1/2}} \Bigg) + \exp\Bigg(\underbrace{\frac{\sqrt{12}}{\sqrt{2}}}_{\lesssim 1} \underbrace{x}_{< -\vert\varepsilon\vert^{-1/2}} \Bigg) \\
            &\lesssim \vert\varepsilon\vert^{3/2}.
        \end{align}
    \end{enumerate}

    Thus $I_{\varepsilon} \lesssim \vert\varepsilon\vert^{3/2}$ uniformly over $x \in \bbR$ and the result follows.
\end{proof}

\begin{lem} \label{lem:Q_holo}
    Let $\varepsilon_0 > 0$ be as in Lemma \ref{lem:three_real_solutions}. Then the function $\Qholnoeps: B'(\varepsilon_0) \cap (0, \infty) \rightarrow L^{\infty}(\bbR)$, defined by 

    \begin{equation}
        \Qhol(x) := \varepsilon^2 \Qren(\varepsilon x), \quad x \in \bbR, \label{eq:Qhol}
    \end{equation}

    has a holomorphic extension to all of $B(\varepsilon_0)$ satisfying 

    \begin{equation}
        \sup_{x \in \bbR} \vert \Qhol(x) - \Qhollim(x) \vert \lesssim \varepsilon^{3/2}.
    \end{equation}
\end{lem}
\begin{proof}
    Firstly note that 

    \begin{align}
        \Qhol(x) &= \varepsilon^2 \Qren(\varepsilon x) \\
        &= -\varepsilon^2 \fren'(\Pren(\varepsilon x)) + \frac{1}{4} \varepsilon^2 \sren^2 \\
        &= -\varepsilon^2 \Big( -3\Pren(\varepsilon x)^2 + 2(1 + \alpha) \Pren(\varepsilon x) + (3 \varepsilon^{-2} - \alpha)\Big) + \frac{1}{4} \varepsilon^2 \sren^2 \\
        &= 3 \left( \varepsilon \Pren(\varepsilon x) \right)^2 - \varepsilon 2(1 + \alpha) \varepsilon \Pren(\varepsilon x) - (3 - \varepsilon^2 \alpha) + \frac{1}{4} (\varepsilon \sren)^2 \\
        &= 3 \Phol(x)^2 - \varepsilon 2(1 + \alpha) \Phol(x) - (3 - \varepsilon^2 \alpha) + \frac{1}{4} (\varepsilon \sren)^2.
    \end{align}

    Therefore, by Lemma \ref{lem:Phi_holo} 

    \begin{align}
        &\vert \Qhol(x) - \Qhollim(x) \vert \\
        &= \left\vert \left( 3 \Phol(x)^2 - \varepsilon 2 (1 + \alpha) \Phol(x) - 3 + \varepsilon^2 \alpha + \frac{1}{4} \sren^2 \right) - \left( 3 \Phol(x)^2 - 3\right) \right\vert \\
        &\leq 3 \underbrace{\vert \Phol(x)^2 - \Phollim(x)^2 - \varepsilon 2 \frac{1 + \alpha}{3} \Phol(x) \vert}_{=: I_{\varepsilon}} + \varepsilon^2 \vert \alpha \vert + \frac{1}{4} \underbrace{\vert \sren \vert^2}_{\lesssim \vert\varepsilon\vert^4}
       \end{align} 

    By Lemma \ref{lem:Phi_holo}, $\varepsilon \mapsto \Phol \in L^{\infty}(\bbR)$ is holomorphic on $B(\varepsilon_0)$ with $\partial_{\varepsilon} \Phollim = \frac{1+\alpha}{3}$. Hence, $I_{\varepsilon}$ equals the second Taylor remainder of $\Phol^2$ at $\varepsilon = 0$ and is thus $\calO(\varepsilon^2)$ uniformly over $x \in \bbR$. The result follows.
\end{proof}

\begin{proof}[Proof of Proposition \ref{prop:holo_family}]
    Let $\varepsilon \in B(\varepsilon_0)$ be arbitrarily. Then by Lemma \ref{lem:Q_holo}

    \begin{align}
        \left\Vert \Qhol - \Qhollim \right\Vert_{\infty} \lesssim \varepsilon^{3/2}.
    \end{align}

    Dividing by $\varepsilon$ and taking the limit $\varepsilon \rightarrow 0$ shows that the family of multiplication operators $\{ M_{\Qhol} : \varepsilon \in B(\varepsilon_0) \}$ is a bounded-holomorphic family (with vanishing derivative) in the sense of \cite[Ch. VII, §1.2]{katoPerturbationTheoryLinear1995}, i.e. a holomorphic function $B(\varepsilon_0) \rightarrow \calB(L^2(\bbR))$. Hence, by \cite[Prob. VII.1.2.]{katoPerturbationTheoryLinear1995} the family $\{ - \partial_x^2 + \Qhol : \varepsilon \in B(\varepsilon_0) \}$ is a holomorphic family in the sense of Kato.
\end{proof}

\begin{rem}
    One can show the stronger statement that $\{ \Hhol : \varepsilon \in B(\varepsilon_0) \}$ is a holomorphic family of type (A) in the sense of Kato - see e.g. \cite[Ch. VII, §1]{katoPerturbationTheoryLinear1995} - by hand from Lemma \ref{lem:Q_holo}. However, since we will not need this fact, we omit the proof here. The family $\{ \Hhol : \varepsilon \in B(\varepsilon_0) \}$ is not a \textit{self-adjoint holomorphic family} (in the sense of \cite{katoPerturbationTheoryLinear1995}) though, since $\sH^{\hol}_{\overline{\varepsilon}} \neq (\sH^{\hol}_{\varepsilon})^{\ast}$ for $\varepsilon \in B(\varepsilon_0) \setminus \bbR$, but this is also not needed for our purposes.
\end{rem}

\begin{proof}[Proof of Theorem \ref{thm:sigma_0}]
\label{page:point}
    Fix $\varepsilon \in B'(\varepsilon_0) \cap (0, \infty)$ and let $\lambda \in \sigma_0(\Lren)$ be arbitrary. We firstly note that $0 \in \sigma_0(\Lren)$ and that it is the largest eigenvalue of $\Lren$: By Theorem \ref{thm:wave_data}, Equation \eqref{eq:ren_AC_skeleton} admits a travelling wave solution with profile $\Pren$. Hence, by translational symmetry, $\Pren' \in \ker \Lren$. For $\varepsilon_0$ small enough, $0 < \inf \sigma_{\ess}(\Hren)$ by Proposition \ref{prop:Schrödinger} and thus $\rho_{\varepsilon} \Pren' \in \ker \Hren$ by Equation \eqref{eq:Schrodinger_eigenfunction}. Because $\Pren$ is monotone, $\rho_{\varepsilon} \Pren'$ is strictly positive and thus by \cite[Thm. 9.40]{teschlMathematicalMethodsQuantum2014}, $0$ is the smallest eigenvalue of the Schrödinger operator $\Hren$ and thus the largest of $\Lren$.

    Now let $\lambda \in \sigma_0(\Lren) \setminus \{ 0 \}$. If $\Real(-\lambda) \geq \inf \sigma_{\ess}(\Hren)$, then by Proposition \ref{prop:Schrödinger}, the condition in \eqref{eq:sigma_0_statement} is trivially satisfied.

    If $\Real(-\lambda) < \inf \sigma_{\ess}(\Hren)$, then $- \lambda \in \sigma_{\pt}(\Hren)$ by Proposition \ref{prop:eigenvector_of_Schrödinger}, and $- \varepsilon^2 \lambda \in \Hhol$ by Lemma \ref{lem:scaling_Hren_Hhol}. Since $\{ \Hhol : \varepsilon \in B(\varepsilon_0) \}$ is a holomorphic family in the sense of Kato by Proposition \ref{prop:holo_family}, by analytic perturbation theory (e.g. \cite[Thm. XII.8]{reedMethodsModernMathematical1978IV}, \cite[Thm. VII.1.8.]{katoPerturbationTheoryLinear1995}) the function $\varepsilon \mapsto - \varepsilon^2 \lambda$ is holomorphic on $B(\varepsilon_0)$ and thus bounded. Therefore, $\lambda \simeq -\varepsilon^{-2} + \calO(\varepsilon^{-1})$. In particular, this is true for the second-lowest eigenvalue of $\Hren$ (which is positive by the arguments above), i.e. the second-largest of $\Lren$, which concludes the proof. 
\end{proof}

\appendix

\section{Auxiliary lemmas}

\begin{lem} \label{lem:conjugate_lin_op}
    Let $X$ and $Y$ be complex Banach spaces, let $\sL: X \supseteq \calD(\sL) \rightarrow X$ be a densely defined and closed linear operator on $X$, and let $T: X \rightarrow Y$ be a linear isomorphism. Then $T\sL T^{-1}: Y \supseteq T\calD(\sL) \rightarrow Y$ is a densely defined and closed linear operator on $Y$ and (in the notation of \eqref{eq:sigma_not_F} - \eqref{eq:sigma_ess})

    \begin{equation}
        \sigma_{\not F}(\sL) = \sigma_{\not F}(T\sL T^{-1}), \quad \sigma_{k}(\sL) = \sigma_{k}(T\sL T^{-1}), \quad k \geq 0, \quad \sigma_{\ess}(\sL) = \sigma_{\ess}(T\sL T^{-1}).
    \end{equation}
\end{lem}
\begin{proof}
    Since $T$ is a homeomorphism and $\calD(\sL)$ is dense, so is $T\calD(\sL)$. To see that $T\sL T^{-1}$ is closed, note that its graph is the image of the graph of $\sL$ under the isomorphism $T \oplus T: X \oplus X \rightarrow Y \oplus Y$. 

    Now, let $\lambda \in \bbC$. Then conjugation preserves

    \begin{enumerate}
        \item membership of the resolvent set, because $T(\sL - \lambda I) T^{-1} = T\sL T^{-1} - \lambda I$,
        \item closedness of the range, because 

        \begin{equation}
            \range(T\sL T^{-1}) = (T\sL T^{-1})(T\calD(\sL )) = T (\sL \calD(\sL )) = T\range(\sL ),
        \end{equation}
        \item the Fredholm property and index, because 
        
        \begin{align}
                \ker(T\sL T^{-1} - \lambda I) &= T \ker(\sL - \lambda I) \\
                Y / \range(T\sL T^{-1} - \lambda I) &= TX / T \range(\sL - \lambda I) = T(X/\range(\sL - \lambda I))
            \end{align}
    
            and thus
            
            \begin{align}
                \dimension(\ker(\sL - \lambda I)) &= \dimension(\ker(T\sL T^{-1} - \lambda I)), \\
                \codimension(\range(\sL - \lambda I)) &= \codimension(\range(T\sL T^{-1} - \lambda I)).
            \end{align}
    \end{enumerate}

    The result follows.
\end{proof}

For the travelling wave solutions of the Allen--Cahn equations it is well known that the profile decays exponentially fast to its asymptotic states \cite[Sec. 8]{sattingerStabilityWavesNonlinear1976}. Since we want to apply regular perturbation theory and thus need to allow $\varepsilon$ to be complex, the potential $\Qhol$ will involve the wave profile $\Phux$ evaluated at a complex argument. Therefore, we first need to extend the exponential decay of the classical wave to a complex neighborhood of the real axis.

\begin{figure}[htbp]
    \centering
    \begin{minipage}[t]{0.62\textwidth}
        \centering
        \includegraphics[width=\textwidth]{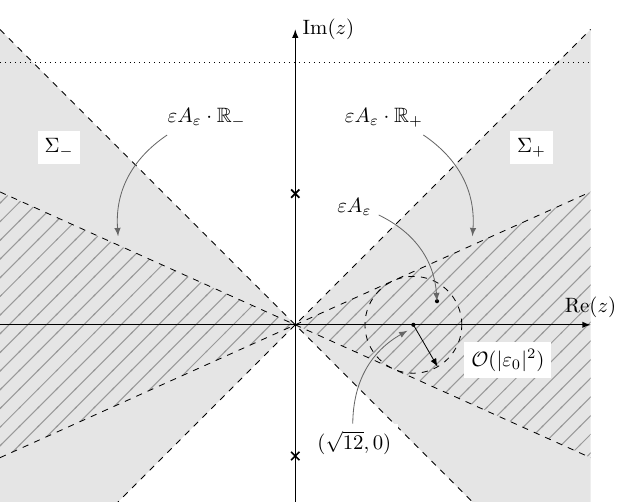} 
        \caption{Regions $\Sigma^+$ and $\Sigma^-$ in the complex plane.}
        \label{fig:sigma_regions}
    \end{minipage}%
    \hfill
    \begin{minipage}[t]{0.38\textwidth}
        \centering
        \includegraphics[width=\textwidth]{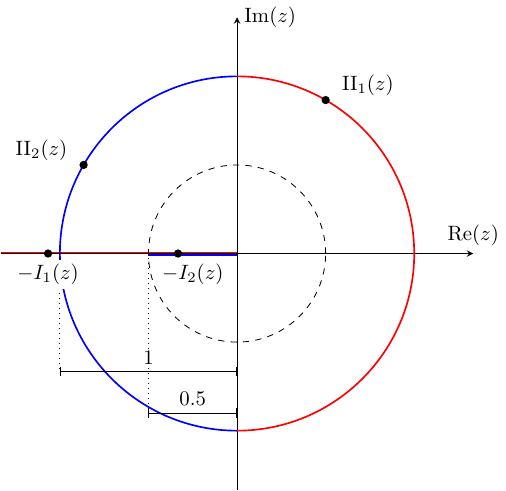} 
        \caption{The blue (case 1) and red (case 2) regions are uniformly separated, respectively.}
        \label{fig:case_distinction}
    \end{minipage}
\end{figure}

\begin{lem} \label{lem:exp_decay}
    Let $\Sigma^+, \Sigma^-, \Sigma$ be defines as in \eqref{eq:sigma_plus}, \eqref{eq:sigma_minus}, \eqref{eq:sigma_full}. There exists a neighborhood $\calN$ of $\Sigma \subseteq \bbC$ such that $\Phux$ (as defined in \eqref{eq:huxley_quant}) has a holomorphic extension to $\calN$ satisfying $\Phux, \Phux' \in L^{\infty}(\calN)$ and

    \begin{align}
        \vert \Phux(z) - 1 \vert &\lesssim \exp\left( - \frac{\Real(z)}{\sqrt{2}} \right), \quad z \in \Sigma^+, \label{eq:exp_decay_plus} \\
        \vert \Phux(z) - 0 \vert &\lesssim \exp\left( \frac{\Real(z)}{\sqrt{2}} \right), \quad z \in \Sigma^-. \label{eq:exp_decay_minus}
    \end{align}
\end{lem}
\begin{proof}
    Notice that

    \begin{equation}
        \Phux(z) = \left( 1 + \exp\left(-z/\sqrt{2}\right) \right)^{-1}
    \end{equation}
    
    (and thus $\Phux'$) is $\sqrt{8}\pi$-periodic in $\Ima(z)$ with fundamental domain $\bbR \times [0, \sqrt{8} \pi)$ and a single pole at $z = \ii \sqrt{2} \pi$, which lies in the interior of the region - see Figure \ref{fig:sigma_regions}. Therefore, since the closure of $\Sigma$ does not contain the pole, $\Phux$ and its derivative are holomorphic and thus locally bounded on a neighborhood $\calN$ of $\Sigma$. Global boundedness follows from the decay estimates \eqref{eq:exp_decay_plus} and \eqref{eq:exp_decay_minus}, which we show now.  
    
    Let $z \in \Sigma^-$ be arbitrary. Then

    \begin{align}
        \vert \Phux(z) - 0\vert &= \left\vert \left( 1 + \exp\left( - \frac{z}{\sqrt{2}} \right) \right)^{-1} \right\vert \\
        &= \exp\left( \frac{\Real(z)}{\sqrt{2}} \right) \Bigg\vert \Bigg( \underbrace{\exp\left( \frac{\Real(z)}{\sqrt{2}} \right)}_{=: I(z)} + \underbrace{\exp\left( - \ii \frac{\Ima(z)}{\sqrt{2}} \right)}_{=: II(z)} \Bigg)^{-1} \Bigg\vert.
    \end{align}

    We distinguish two cases - see Figure \ref{fig:case_distinction}: 
    
    \begin{enumerate}
        \item If $\vert \Ima(z)/\sqrt{2} \vert < \pi/2$, then $II(z)$ lies on the right hemisphere of the unit circle. Hence, the distance to $-I(z)$, which lies on the negative real half-axis, is lower bounded by $1$.
        \item If $\vert \Ima(z)/\sqrt{2} \vert \geq \pi/2$, then since $z \in \Sigma^-$, we have $\Real(z)/\sqrt{2} \leq - \vert \Ima(z)/\sqrt{2} \vert \leq - \pi/2$, and thus $I(z) \leq \exp(-\pi/2) \approx 0.21$. Since $II(z)$ lies on the unit circle, the distance to $-I(z)$ which lies in the $1/2$ unit ball around $0$, is lower bounded by $1/2$.
    \end{enumerate}
    
    Thus, the term in the outer brackets is lower bounded by $1/2$, and hence the entire right-hand side is upper bounded by $2 \exp\left( \Real(z)/\sqrt{2} \right)$.

    Let $z \in \Sigma^+$ be arbitrary and recall the definition of $\Phux$ from \eqref{eq:huxley_quant}. Then

    \begin{align}
        \vert \Phux(z) - 1\vert &= \left\vert \left( 1 + \exp\left( - \frac{z}{\sqrt{2}} \right) \right)^{-1} - 1\right\vert = \left\vert \frac{-\exp\left( - \frac{z}{\sqrt{2}} \right)}{1 + \exp\left( - \frac{z}{\sqrt{2}} \right)} \right\vert \\
        &= \exp\left( - \frac{z}{\sqrt{2}} \right) \underbrace{\Bigg\vert \Bigg( 1 + \exp\left( - \frac{z}{\sqrt{2}} \right) \Bigg)^{-1} \Bigg\vert}_{= \Phux(z)}
    \end{align}

    On the one hand, since $\Phux$ is holomorphic on a neighborhood of $\calN$, it is bounded on $\Sigma^+ \cap \{ \Real(z) \leq 1 \}$. On the other hand, on the set $\Sigma^+ \cap \{ \Real(z) \geq 1 \}$, by the reverse triangle inequality 

    \begin{equation}
        \left\vert 1 + \exp\left( - \frac{z}{\sqrt{2}} \right) \right\vert \geq \left\vert 1 - \left\vert \exp\left( - \frac{z}{\sqrt{2}} \right) \right\vert \right\vert \geq 1 - \exp\left( - \frac{1}{\sqrt{2}} \right) \approx 0.51 > 0.
    \end{equation}

    Hence $\Phux$ is bounded uniformly on all of $\Sigma^+$ and the claim follows. 
\end{proof}

\begin{lem} \label{lem:existence_of_holo_sol}
    Let $\calU$ be open, $\varepsilon_0 \in \calU$, and $b_0, \ldots, b_n: \calU \rightarrow \bbC$ be holomorphic such that the polynomial $f_{\varepsilon_0}(z) = \sum_{k = 0}^n b_k(\varepsilon_0) z^k$ has $l$ simple solutions. Then there exist an open neighborhood $\calU_0 \subseteq \calU$ containing $\varepsilon_0$ and holomorphic functions $z_1, \ldots, z_l: \calU_0 \rightarrow \bbC$ such that for every $\varepsilon \in \calU_0$ 

    \begin{equation}
        f_{\varepsilon}(z_k(\varepsilon)) = 0 \quad \forall k \in \{ 1, \ldots, l \}.
    \end{equation}
\end{lem}
\begin{proof}
    Label the $l$ distinct elements in the set $f^{-1}_{\varepsilon_0}(0)$ as $\{ s_1(\varepsilon_0), \ldots, s_l(\varepsilon_0) \}$ and let $i \in \{ 1, \ldots, l \}$ be arbitrary. Then the map 

    \begin{equation}
        f: \calU \times \bbC \rightarrow \bbC; \quad (\varepsilon, z) \mapsto f_{\varepsilon}(z) = \sum_{k = 0}^n b_k(\varepsilon) z^k
    \end{equation}

    is holomorphic, satisfies $f_{\varepsilon_0}(s_i(\varepsilon_0)) = 0$, and, since $s_i(\varepsilon_0)$ is a simple root, $\partial_z f_{\varepsilon_0}(s_i(\varepsilon_0)) \neq 0$. Thus, by the (complex analytic) implicit function theorem \cite[App. II, Lem. 3B(f)]{whitneyComplexAnalyticVarieties1972}, there exists an open neighborhood $\calU_0^i \subseteq \calU$ of $\varepsilon_0$ and a unique holomorphic function $z_i: \calU_0^i \rightarrow \bbC$ such that 

    \begin{equation}
        f_{\varepsilon}(z_i(\varepsilon)) = 0 \quad \forall \varepsilon \in \calU_0^i.  
    \end{equation}

    A common domain for all $z_i$ can be obtained by taking $\calU_0 := \bigcap_{i = 1}^l \calU_0^i$.
\end{proof}

\begin{lem} \label{lem:closed_lin_op}
    If $a_0, a_1 \in L^{\infty}(\bbR)$, then the differential operator
        
    \begin{equation}
        \sA = \partial_x^2 + a_1(x) \partial_x + a_0(x). \label{eq:diff_op}
    \end{equation}

    with domain $H^2(\bbR)$ is a closed linear operator on $L^2(\bbR)$ with non-empty resolvent set. 
\end{lem}
\begin{proof}
    Conjugating via the Fourier transform to a multiplication operator shows that $(\partial^2, H^2(\bbR))$ is a closed linear operator on $L^2(\bbR)$. By the Gagliardo--Nirenberg inequality \cite[Thm. 5.2]{adamsSobolevSpaces2003}, we have 

    \begin{equation}
        \left\Vert a_1 \partial_x u + a_0 u \right\Vert_{L^2} \lesssim \max_{k = 0, 1} \Vert a_k \Vert_{L^{\infty}} (\varepsilon \Vert \partial_x^2 u \Vert_{L^2} + C_{\varepsilon,k} \Vert u \Vert_{L^2})
    \end{equation}

    Thus, choosing $\varepsilon > 0$ small enough gives a relative bound $\tilde{\varepsilon} < 1$ (and corresponding $\tilde{C}_{\tilde{\varepsilon}}$) of $a_1 \partial_x + a_0$ with respect to $\partial^2$. Hence by the Kato--Rellich theorem \cite[Thm. IV.1.1.]{katoPerturbationTheoryLinear1995} $\sA$ is closed on $H^2(\bbR)$. 

    Let $\calR(\partial_x^2, \lambda)$ denote the resolvent of $\partial_x^2$. Since $\sigma(\partial_x^2) = [0, \infty)$, we may choose $\varepsilon > 0$ small enough and $\lambda_{\varepsilon} < 0$ large enough such that 

    \begin{equation}
        \tilde{C}_{\tilde{\varepsilon}} \Vert \calR(\partial_x^2, \lambda_{\varepsilon}) \Vert + \tilde{\varepsilon} \underbrace{\Vert \partial_x^2 \calR(\partial_x^2, \lambda_{\varepsilon}) \Vert}_{< \infty} < 1.
    \end{equation}

    Then by \cite[Thm. IV.3.17.]{katoPerturbationTheoryLinear1995} the resolvent set of $\sA$ contains $\lambda_{\varepsilon}$ and is thus non-empty.    
\end{proof}

\begin{lem} \label{lem:compact_multiplication}
    Let $a \in L^{\infty}(\bbR)$ such that $\lim_{R \rightarrow \infty} \sup_{\vert x \vert > R} \vert a(x) \vert = 0$. Then the multiplication operator $M_a: H^1(\bbR) \rightarrow L^2(\bbR); u \mapsto au$ is compact.
\end{lem}
\begin{proof}
    Let $(u_n)_{n \in \bbN} \subseteq H^1(\bbR)$ be a bounded sequence and let $\varepsilon > 0$ be arbitrary. By choice of $a$ and because $(u_n)_{n \in \bbN}$ is also bounded in $L^2(\bbR)$ we may choose $R > 0$ large enough such that 

    \begin{equation}
        \int_{\vert x \vert > R} \vert a(x) u_n(x) \vert^2 dx \leq \sup_{\vert x \vert > R} \vert a(x) \vert^2 \underbrace{\int_{\vert x \vert > R} \vert u_n(x) \vert^2 dx}_{\leq \Vert u_n \Vert_{L^2(\bbR)}^2 \lesssim 1} < (\varepsilon/2)^{1/2}, \quad \forall n \in \bbN.
    \end{equation}

    The restriction $(u_n \vert_{(-R,R)})_{n \in \bbN}$ is also a bounded sequence in $W^{1,2}((-R,R))$. By the Rellich--Kondrachov theorem \cite[Thm. 6.3.II]{adamsSobolevSpaces2003} there exists a subsequence $(u_{n_k})_{k \in \bbN}$ such that the sequence $(u_{n_k} \vert_{(-R,R)})_{k \in \bbN}$ is Cauchy. Thus, we may choose $N$ large enough such that for every $k,l \geq N$

    \begin{equation}
        \Vert M_{a} u_{n_k} - M_{a} u_{n_l} \Vert_{L^2(\bbR)} \leq \underbrace{\sup_{x \in \bbR} \vert a(x) \vert}_{\lesssim 1} \Vert u_{n_k} - u_{n_l} \Vert_{L^2((-R,R))} + \Vert a u_{n_k} - a u_{n_l} \Vert_{L^2(\bbR \setminus [-R,R])} < \varepsilon/2 + \varepsilon/2 = \varepsilon.
    \end{equation}

    Hence $(M_a u_{n_k})_{k \in \bbN} \subseteq (M_a u_n)_{n \in \bbN}$ is a Cauchy subsequence in $L^2(\bbR)$, which thus converges in $L^2(\bbR)$. 
\end{proof}

\printbibliography 

\end{document}